\documentclass[12pt,halfline,a4paper]{ouparticle}
\usepackage{amsmath, amsfonts, amsthm, multicol, hyperref, amssymb, indentfirst}
\theoremstyle{definition}
\newtheorem{definition}{Definition}[section]
\newtheorem{example}{Example}[definition]
\newtheorem{theorem}{Theorem}[section]

\newtheorem{prop}[theorem]{Proposition}
\newtheorem{lemma}[theorem]{Lemma}

\begin{document}

\title{Kac-Moody Quaternion Lie Algebra}

\author{
\name{Ferdi, Amir Kamal Amir, Andi Muhammad Anwar}
\address{Department of Mathematics\\ Faculty of Mathematics and Natural Sciences\\
Hasanuddin University,\\ Makassar 90245, Indonesia}}

\abstract{\quad This research aims to define Kac-Moody Lie algebra in Quaternion by using the concept of Quaternification of Lie algebra. The results of this research obtained the definition of Universal Kac-Moody Quaternion Lie algebra, Standard Kac-Moody Quaternion Lie algebra, and Reduced Kac-Moody Quaternion Lie algebra}

\date{}
\maketitle

\renewcommand\qedsymbol{$\blacksquare$}
\section{Introduction}
\label{sec1}
Algebra is a branch of mathematics that studies structures, relations, and mathematical operations involving abstract objects such as numbers, variables, and arithmetic operations. One of the topics in algebra that is widely researched is Lie algebra. The term Lie algebra is taken from the name of the Norwegian mathematician Marius Sophus Lie (1842–1899). Shopus Lie developed Lie algebra to study the concept of infinitesimal transformations in the 1870s. Several researchers have published their writings in the form of books and papers on Lie algebra, such as Gerard G.A. Bauerle and Eddy A. De Kerf \cite{ref2}, Brian C. Hall \cite{ref7}, dan James E. Humpherys \cite{ref8}. In their books, they have discussed some material about Lie algebra, such as basic concepts about Lie algebra, homomorphisms of Lie algebra, complexification of real Lie algebra, simple Lie algebra, semisimple Lie algebra, and others. Furthermore, in 1984, Rolf Farnsteiner \cite{ref5} began to introduce Lie algebra in quaternion by investigating the Lie algebra that is isomorphic to the central quotient of the division algebra of quaternion, which is then referred to as Lie algebra quaternion. Then, Dominic Joyce \cite{ref9} gave the definition of Lie algebra Quaternion as an AH module object that satisfies the bracket condition of Lie algebra. The AH module is a more specialized concept than the H module. Furthermore, Tasioki Kori \cite{ref13} published papers that discuss Quaternification on complex Lie algebra, which is an extension of complexification on real Lie algebra. In the paper, Tasioki Kori introduced the definitions of Lie algebra quaternion, quaternification of Lie algebra, quaternification on simple Lie algebra, and quaternification on complex Lie algebra. root space decomposition of a Quaternion Lie algebra. One of the important topics in the development of Lie algebra is Kac-Moody Lie algebra.Kac-Moody Lie algebra is a Lie algebra that uses Cartan matrix generalization. Kac-Moody Lie algebra was proposed with two different properties by Victor Gershevich Kac \cite{ref10} and Robert Vaughan Moody \cite{ref17}. Since there are two definitions with different properties, Steven Berman \cite{ref3} gave names for each property, namely Standard Kac-Moody Lie Algebra and Reduced Kac-Moody Lie Algebra. However, to get the two definitions, first define the Universal Kac-Moody Lie Algebra, which is then applied with two properties so as to obtain Standard Kac-Moody Lie Algebra and Reduced Kac-Moody Lie Algebra.Based on the description above, the interesting thing to study is the combination of quaternification on complex Lie algebra with Lie Kac-Moody algebra. Therefore, this study discusses the construction of Universal, Standard, and Reduced quaternion Kac-Moody Lie algebras using quaternification.

\section{Construction of Universal Kac-Moody Quaternion Lie Algebra}
\label{sec2}

\subsection{General Structure Over Quatenion}
\label{sec2.1}

This section will discuss the quaternion tensor algebra \(T(V)\) constructed from the quaternion module \((\mathcal{V},J)\), the universal quaternion enveloping algebra \(U(L)\) of a quaternion Lie algebra, and finally the quaternion algebra \(L(X,JX)\) generated by the basis \(\{X,JX\}\). The discussion will begin with the definition of the quaternion tensor algebra \(T(V)\). 

Let \((\mathcal{V},J)\) be a module over \(\mathbb{H}\). Let \(\sigma\) be a linear involution over \(\mathbb{C}\) on \(\mathcal{V}\) that is anti-commutative with \(J:J\sigma=-\sigma J\), and let \(\tau\) be a complex conjugation linear involution on \(\mathcal{V}\) that commutes with \(\sigma\tau=\tau\sigma\). Let \(\mathcal{V}_0\) be the eigenspace of \(\sigma\) corresponding to the eigenvalue \(+1\). Then, \(\mathcal{V}=\mathcal{V}_0+J\mathcal{V}_0\). Both \(\mathcal{V}_0\) and \(J\mathcal{V}_0\) are invariant under \(\tau\).
\begin{definition}
    Let \(T(V)\) be a \(\mathbb{R}\)-submodule of the \(\mathbb{H}\)-module \((\mathcal{V},J)\). \(T(V)\) is called \textbf{quaternion tensor algebra} if \(T(V)\) that satisfies the following property:
    \begin{enumerate}
        \item \(T(V)\) is a real tensor algebra
        \item \(\sigma\) and \(\tau\) are homomorphism of tensor algebra \(T(V)\)
        \begin{enumerate}
            \item \(T(V)\) is invariant under the involutions \(\sigma\) and \(\tau\).
            \item\begin{multicols}{2}
            \(\sigma(t_1+t_2)=\sigma(t_1)+\sigma(t_2)\)\\\(\sigma(t_1\cdot t_2)=\sigma(t_1)\cdot\sigma(t_2)\)\\\(\sigma(a\cdot t_1)=a\cdot\sigma(t_1)\)\\\(\sigma(1)=1\)\\for all \(a\in \mathbb{R}\) and \(t_1,t_2\in T(V)\)
            \columnbreak
            
            \(\tau(t_1+t_2)=\tau(t_1)+\tau(t_2)\)\\\(\tau(t_1\cdot t_2)=\tau(t_1)\cdot\tau(t_2)\)\\\(\tau(a\cdot t_1)=a\cdot\tau(t_1)\)\\\(\tau(1)=1\)
            \end{multicols}
        \end{enumerate}
    \end{enumerate}
\end{definition}
It can be seen that multiplication in \(T(V)\) is associative. Furthermore, \(1\in \mathbb{H}\) is the identity element of \(T(V)\). This indicates that \(T(V)\) is an associative quaternion algebra with an identity element. A tensor quaternion algebra can also be referred to as a tensor algebra generated by the quaternion module \((\mathcal{V},J)\).

For the quaternion tensor algebra \(T(V)\), the subspaces \((T(V))^+\) and \((T(V))^-\) are denoted as the eigenspaces of the involution \(\sigma\) with eigenvalues \(+1\) and \(-1\), respectively. \((T(V))^+\) is a vector subspace of \(T(V)\) that is invariant under complex conjugation \(\tau\), and
    \begin{equation}\label{}
        T(V)=(T(V))^++(T(V))^-, (T(V))^+=T(V)\cap \mathcal{V}_0, (T(V))^-=T(V)\cap J\mathcal{V}_0
    \end{equation}    
where \((T(V))^+\) is a subalgebra of \(T(V)\). 

Next, the quaternification of the tensor algebra will be defined as follows:
\begin{definition}
    Let \(T_0(V)\) be a real or complex tensor algebra. Let \(T(V)\) be a quaternion tensor algebra. \(T(V)\) is said to be the quaternification of \(T_0(V)\) if \(T_0(V)\) is a real tensor subalgebra of \((T(V))^+\), and if there exists a real vector subspace \(b(V)\) of \((T(V))^-\) such that \(T_0(V)+b(V)\) generates \(T(V)\) as a real tensor algebra.
\end{definition}
Next, we will define the associative quaternion algebra.

Let \((\mathcal{V'},J)\) be a module over \(\mathbb{H}\). Let \(\sigma\) be a linear involution over \(\mathbb{C}\) on \(\mathcal{V'}\) that is anti-commutative with \(J:J\sigma=-\sigma J\), and let \(\tau\) be a complex conjugation linear involution on \(\mathcal{V'}\) that commutes with \(\sigma\tau=\tau\sigma\). Let \(\mathcal{V'}_0\) be the eigenspace of \(\sigma\) corresponding to the eigenvalue \(+1\). Then, \(\mathcal{V'}=\mathcal{V'}_0+J\mathcal{V'}_0\). Both \(\mathcal{V'}_0\) and \(J\mathcal{V'}_0\) are invariant under \(\tau\).
\begin{definition}
        Let \(A\) be a \(\mathbb{R}\)-submodule of the \(\mathbb{H}\)-module \((\mathcal{V'},J)\). \(A\) is called \textbf{quaternion algebra} if \(A\) that satisfies the following property:
    \begin{enumerate}
        \item \(A\) is a real algebra
        \item \(\sigma\) and \(\tau\) are homomorphism of algebra \(A\)
        \begin{enumerate}
            \item \(A\) is invariant under the involutions \(\sigma\) and \(\tau\).
            \item\begin{multicols}{2}
            \(\sigma(u+v)=\sigma(u)+\sigma(v)\)\\\(\sigma(uv)=\sigma(u)\cdot\sigma(v)\)\\\(\sigma(au)=a\sigma(u)\)\\\(\sigma(1)=1\)\\for all \(a\in \mathbb{R}\) and \(u,v\in A\)
            \columnbreak
            
            \(\tau(u+v)=\tau(u)+\tau(v)\)\\\(\tau(uv)=\tau(u)\tau(v)\)\\\(\tau(a u)=a\tau(u)\)\\\(\tau(1)=1\)
            \end{multicols}
        \end{enumerate}
    \end{enumerate}
\end{definition}
An algebra \(A\) is called an associative quaternion algebra if for every \(u,v,w\in A\) the following holds:
\begin{equation}
    (uv)ww=u(vw)
\end{equation}
An associative quaternion algebra has an identity element \(e\) if \(eu=ue=u\) for every \(u \in A\).

An associative quaternion algebra can also be refreed to as an associative algebra generated by the quaternion module \((\mathcal{V'},J)\).

For an associative quaternion algebra \(A\), let \(A^pm\) denote the eigenspaces of the involution \(\sigma\) with eigenvalues \(\pm{1}\), respectively. \(A^pm\) are vector subspaces of \(A\) that are invariant under complex conjugation \(\tau\), and
\begin{equation}
    A=A^++A^-, A^+=A\cap \mathcal{V'}_0, A^-=A\cap J\mathcal{V'}_0.
\end{equation}
where \(A^+\) is a subalgebra of \(A\).

Next, the quaternification of an associative algebra will be defined as follows:
\begin{definition}
    Let \(A_0\) be a real or complex tensor algebra. Let \(A\) be a quaternion tensor algebra. \(A\) is said to be the quaternification of \(A_0\) if \(A_0\) is a real tensor subalgebra of \(A^+\), and if there exists a real vector subspace \(b\) of \(A^-\) such that \(A_0+b\) generates \(A\) as a real tensor algebra.
\end{definition}
\begin{definition}
    Let \(A_1\) and \(A_2\) be a quaternion associative algebra. A homomorphism \(\varphi:A_1\to A_2\) of real associative algebra is called a \textbf{homomorphism of quaternion associative algebra} if
    \begin{equation}\label{1}
        \varphi(\sigma u)=\sigma\varphi(u) \text{ and } \varphi(\tau u)=\tau\varphi(u), \text{ for all } u\in A_1
    \end{equation}
\end{definition}
\begin{definition}
    Let \(A\)  be a quaternion associative algebra and let \(I\) be an ideal of \(A\) viewed as a real associative algebra. \(I\) is called  \textbf{ideal of quaternion associative algebra} if \(I\) invariant over involution \(\sigma\).
\end{definition}
The quotient space of the associative quaternion algebra \(A\) by an ideal \(I\) is equipped with the structure of an associative quaternion algebra, where the involution \(\Hat{\sigma}\) on \(A/I\) is defined by
\begin{equation}
    \Hat{\sigma}(x+I)=\sigma x+I
\end{equation}
For a homomorphism of associative quaternion algebras \(\phi:A_1\to A_2\), the kernel \(\ker \phi\) is an ideal of \(A_1\).

From the associative quaternion algebra, we naturally obtain a quaternion Lie algebra \(A_L\) (Kori, 2023) by defining the Lie bracket as follows:
\begin{equation}
    [c_1\otimes u, c_2\otimes v]=(c_1c_2)\otimes(u\cdot v)-(c_2c_1)\otimes(v\cdot u)
\end{equation}
for every \(c_1,c_2\in \mathbb{H}\) and \(u,v\in A_0\), where \(A_0\) is quaternification of \(A\). The associativity of \(A\) leads to the Jacobi identity of this bracket.

Next, it will be shown that the quaternion tensor algebra \(T(V)\) is general. Suppose \(\phi:V\to U\) is a linear map from the quaternion module \((\mathcal{V},J)\) to associative quaternion algebra \(U\). Then there exist a unique homomorphism (of associative quaternion algebras) \(\psi:T(V)\to U\) such that \(\phi=\psi\circ\varphi\), where \(\varphi\) is the canonical embedding of \(V\) into \(T(V)\). This embedding is given by \(\varphi(t_1):t^1\) (where \(t^1\in T(V)\)), we have
\begin{equation}
    t=t^0+t^1+\dots+t^k+\dots, \text{  } (t^k\in T^kV)
\end{equation}
the definition of \(\psi\) is simply given by
\begin{equation}
    \psi\Bigg(\sum_{k=0}^{\infty} \sum_{i_1\dots i_k} a_{i_1\dots i_k}v_{i_1}\otimes\dots\otimes v_{i_k}\Bigg)=\sum_{k=0}^{\infty} \sum_{i_1\dots i_k} a_{i_1\dots i_k}\phi(v_{i_1})\otimes\dots\otimes \phi(v_{i_k})
\end{equation}
where the product on the right-hand side is the product in the associative quaternion algebra \(U\).

The next topic is the definition and construction of the general enveloping algebra \((U(L),i)\) of the quaternion Lie algebra \(L\). Subsequently, we will see how the ideas from the previously discussed associative quaternion tensor algebra are used in the construction of \(U(L)\).
\begin{definition}
    Let \(L\) be a quaternion Lie algebra. A \textbf{universal enveloping algebra} of \(L\) is a pair \((U(L),i)\) with \(U(L)\) a quaternion associative algebra with unit element and homomorphism \(i:L\to U(L)\) be a homomorphism quaternion Lie algebra where \(U(L)\) is considered as a quaternion Lie algebra (because for any quaternion associative algebra has a quaternion Lie algebra in a natural way), i.e. \(i\) is a linear map satisfying 
    \begin{equation}
        i([x,y])=i(x)i(y)-i(y)i(x)\text{ } (x,y\in L) 
    \end{equation}
    \begin{equation}
        i(\sigma x)=\sigma i(x) \text{ and } i(\tau x)=\tau i(x) \text{ } (x\in L)
    \end{equation}
    Futhermore, the pair \((U(L),i)\) is such that for any other pair \((W,j)\) with \(W\) an quaternion associative algebra with unit element and \(j\) a quaternion Lie algebra homomorphism \(j:L\to W\) there exist a unique homomorphism \(\psi_1:U(L)\to W\), mapping the identity of \(U(L)\) to identity of  \(W\), and such that \(j=\psi_1 \circ i\).
\end{definition}
The discussion begins with the quaternion Lie algebra \(L\). Given a quaternion Lie algebra 
\(L\), \(L\) is a submodule of \(\mathbb{R}\) within the quaternion module \((\mathcal{V},J)\). Consequently, one can construct the quaternion tensor algebra \(T(L)\) with an identity element,
\begin{equation}
    T(L)=\mathbb{H}\oplus L\oplus T^2L\oplus T^3L\oplus\dots\oplus T^kL\oplus\dots
\end{equation}
The next step is to construct the ideal \(I\) in \(T(L)\), using the Lie bracket in \(L\). Consider elements in \(L\oplus T^2L\subset T(L)\) of the form
\begin{equation}
    I_{x,y}=x\otimes y-y\otimes x-[x,y] \text{   } (x,y\in L)
\end{equation}
An ideal \(I\) in \(T(L)\) will be constructed form elements of this form, by enforcing the condition that elements of \(I\) are annihilated from both the left and right by element of \(T(L)\). Specifically, \(I\) is given by
\begin{equation}
    I:=\Big\{\sum_{x,y\in L} t\otimes I_{x,y}\otimes t'|x,y\in L;t,t'\in T(L)\Big\}
\end{equation}
with the two sided ideal \(I\), we can define the quaternion Lie algebra quotient
\begin{equation}
    U(L):=T(L)/I
\end{equation}
and the canonical projection
\begin{equation}
    \Psi:t\in T(L)\to \Psi(t)\in U(L)
\end{equation}
by definition, we have
\begin{equation}
    \Psi(I)=0\in U(L)
\end{equation}
In particular, since \(\Psi\) is a canonical homomorphism, it satisfies
\begin{equation}
    \Psi(x\otimes y-y\otimes x-[x,y])=\Psi(x)\Psi(y)-\Psi(y)\Psi(x)-\Psi([x,y])=0
\end{equation}
The restriction of \(\Psi\) to the subspace \(L\) of \(T(L)\) is given by the homomorphism
\begin{equation}
    i\equiv\Psi|_L:L\to U(L)
\end{equation}
Thus, for \(x\) and \(y\) in \(L\), we obtain
\begin{equation}
    i([x,y])=\Psi([x,y])=\Psi(x)\Psi(y)-\Psi(y)\Psi(x)=i(x)i(y)-i(y)i(x)
\end{equation}
The additional condition for the homomorphism of quaternion Lie algebras will be shown next.
Let \(u_1+Jv_1\in L\). We have the following calculations:
\begin{equation}
    i(\sigma_1(u_1+Jv_1))=i(u_1-Jv_1)=u_2-Jv_2=\sigma_2-Jv_2=\sigma_2(i(u_1+Jv_1),
\end{equation}
and
\begin{equation}
    i(\tau_1(u_1+Jv_1))=i(\Bar{u}_1-J\Bar{v}_1)=\Bar{u}_2-J\Bar{v}_2=\tau_2-Jv_2=\tau_2(i(u_1+Jv_1),
\end{equation}
Thus, \(i\) is a homomorphism of quaternion Lie algebras.

Let \((W,j)\) be another enveloping algebra. Define a map \(\psi:T(L)\to W\) by
\begin{equation}
    \psi\Bigg(\sum_{k=0}^{\infty} \sum_{i_1\dots i_k} a_{i_1\dots i_k}v_{i_1}\otimes\dots\otimes v_{i_k}\Bigg)=\sum_{k=0}^{\infty} \sum_{i_1\dots i_k} a_{i_1\dots i_k}j(v_{i_1})\otimes\dots\otimes j(v_{i_k})
\end{equation}
The map \(\psi\) is an associative quaternion algebra homomorphism with an identity element. Since \(j:L\to U\) is a homomorphism of quaternion Lie algebras, it follows that
\begin{equation}
    \psi([x,y]-xy-yx)=j([x,y])-j(x)j(y)-j(y)j(x)=0
\end{equation}
Thus, \(I\subset \ker \psi\), so \(\psi_1:U(L)\to W\) is a well defined homomorphism of assciative quaternion algebras with an identity element. Therefore,
\begin{equation}
    (\psi_1\circ i)(x)=\psi(x)=j(x)
\end{equation}
by linearity, we have
\begin{equation}
    j=\psi_1\circ i
\end{equation}
Assume that \(\psi_1':U(L)\to W\) is another homomorphism of associative quaternion algebras with an identity element such that
\begin{equation}
    j=\psi_1'\circ i
\end{equation}
Then,
\begin{equation}
    \psi_1(x)=j(x)=(\psi_1'\circ i)(x)=\psi_1'(x)
\end{equation}
Therefore, \(\psi_1=\psi_1'\). This shows that \(\psi_1\) is unique as required.

The final structure requires the concept of a quaternion Lie algebra generated by the basis \(\{X,JX\}\). Let \(X\) be a finite set over \(\mathbb{C}\) identified with a subset \(1\otimes X\subset \mathbb{H}\otimes X\). Denote \(JX=j\otimes X\subset \mathbb{H}\otimes X\). The set \((X,JX)\) is a subset of the module \(\mathbb{H}\). Note that the elements of \(\{X,JX\}\) serve as a basis for the quaternion module denoted as \(V\equiv V(X,JX)\). Next, consider the set \(V(X,JX)\) of formal sums given by
\begin{equation}
    V(X,JX)\equiv \Big\{\sum_{i=1}^n [c(x_i)x_i+d(x_i)Jx_i]|c(x_i),d(x_i)\in \mathbb{H}; x_i, Jx_i\in \{X,JX\}\Big\}
\end{equation}
Define the addition of two elements in \(V(X,JX)\) and multiplication by a real number \(\lambda\) as follows:
\[\sum_{i=1}^n [c(x_i)x_i+d(x_i)Jx_i]+\sum_{i=1}^n [p(x_i)x_i+q(x_i)Jx_i]=\sum_{i=1}^n [(c(x_i)+p(x_i)x_i+(d(x_i)+q(x_i))Jx_i]\]
and
\[\lambda\sum_{i=1}^n [c(x_i)x_i+d(x_i)Jx_i]=\sum_{i=1}^n [\lambda c(x_i)x_i+\lambda d(x_i)Jx_i]\]
where \(c(x_i),d(x_i),p(x_i),q(x_i)\in \mathbb{H}, x_i, Jx_i\in \{X,JX\}, \lambda\in \mathbb{R}\)
Thus, \(V(X,JX)\) is a real vector space generated by the set \(\{X,JX\}\). Next, define the operation \(\bullet:\mathbb{H}\times V(X,JX)\to V(X,JX)\) by the rule: \((x,v)\to xv\) which satisfies \(x(yv)=(xy)v\) for all \(x,y\in \mathbb{H}\) and \(v\in V(X,JX)\). Consequently, \(V(X,JX)\) is a quaternion module generated by the basis \(\{X,JX\}\).

The next step involves the construction of the quaternion tensor algebra \(T(V)=T(V(X,JX))\) from the quaternion module \((\mathcal{V},J)\). From the quaternion tensor algebra \(T(V)\), we obtain the quaternion Lie algebra \(T(V)_L\) with the Lie bracket defined by:
\begin{equation}
    [c_1\otimes u,c_2\otimes v]\equiv(c_1c_2)\otimes(u\otimes v)-(c_2c_1)\otimes(v\otimes u)
\end{equation}
for \(c_1,c_2\in \mathbb{H}\) and \(u,v\in T_0(V)\), where \(T_0(V)\) denotes the quaternification of \(T(V)\), Thus, we obtain the Lie algebra \(T(V)_L\). Furthermore, we can obtain the smallest quaternion Lie algebra generated by the set \(\{X,JX\}\). Observe that in \(T(V)_L\), the quaternion Lie subalgebra contains the set \(\{X,JX\}\). The smallest subalgebra containing \(\{X,JX\}\) is the intersection of all subalgebras of \(T(V)_L\) that contain \(\{X,JX\}\). This intersection is the quaternion Lie algebra denoted by \(L(X,JX)\). Thus, \(L(X,JX)\) is called the quaternion Lie algebra generated by the basis \(\{X,JX\}\)

Based on the discussion about the general enveloping algebra, we can demonstrate that \(L(X,JX)\) is non-trivial by constructing a representation \(\psi\) of \(L(X,JX)\).

Let \(\phi\) be a mapping from the set \((X,JX)\) to a quaternion Lie algebra \(U\):
\begin{equation}
    \phi:(X,JX)\to U
\end{equation}
Then, \(\phi\) has a unique extension
\begin{equation}
    \psi:L(X,JX)\to U
\end{equation}
where \(\psi\) is a homomorphism of quaternion Lie algebras.

As an application, consider the set \((X,JX)\) and the mapping \(\phi\) from \((X,JX)\) to the general linear algebra \(gl(W)\) of the quaternion module \((\mathcal{V},J)\):
\begin{equation}
    \phi:(X,JX)\to gl(W)
\end{equation}
This mapping \(\phi\) can be uniquely extended to a homomorphism of quaternion Lie algebras:
\begin{equation}
    \psi:L(X,JX)\to gl(W)
\end{equation}
Thus, \(\psi\) is a representation of \(L(X,JX)\). Consequently, it follows that \(L(X,JX)\) is non-trivial.
\subsection{Realization of Generalized Cartan Matrix Over Quaternion}
\label{sec2.2}
This subsection discusses the concepts of Cartan matrix generalization and realizations. These concepts are utilized in subsection 2.1, where the construction of quaternion Lie algebras involves the use of generalized Cartan matrices.

Based on the concept of Cartan matrix generalization, one can construct semisimple Lie algebras of finite dimension, as demonstrated by Serre. However, for Serre's construction, it is required that the generalized Cartan matrices used are non-singular. Therefore, the concept of realization is needed to ensure that the generalization of Cartan matrices is non-singular. Thus, in constructing semisimple quaternion Lie algebras, the same approach as Serre's construction will be applied, utilizing the concept of realization for generalized Cartan matrices specifically in the context of quaternions.

Let \(A_0\) be an \(n\times n\) matrix over \(\mathbb{C}\) with rank \(A_0=r\). The realization of matrix \(A_0\) is \(\{H, \Pi, \Pi^V\}\) where:
\begin{enumerate}
    \item \(H\) is a complex vector space with dimension \(2n-r\).
    \item \(\Pi^V=\{\alpha_1^v,\alpha_2^v,\dots, \alpha_n^v\}\) is a set of \(n\) independent elements in \(H\).
    \item \(\Pi=\{\alpha_1,\alpha_2,\dots, \alpha_n\}\) is a set of \(n\) independent elements in the dual space \(H^*\) of \(H\).
    \item The dual contraction between \(H^*\) and \(H\) satisfies \((A_0)_{ij}=\langle\alpha_j,\alpha_i^v\rangle\).
\end{enumerate}
Next, the realization of matrix \(A_0\) is extended to the realization of \(A\), where \(A\) is the quaternionification of \(A_0\). Here is the definition of the realization of matrix \(A\):
\begin{definition}
Let \(A\) be an \(n\times n\) quaternion matrix over \(R\) with rank \(A=r\). The matrix \(A = A_0+JB_0\), where \(A\) is the quaternification of \(A_0\). \(A_0\) and \(B_0\) are \(n\times n\) matrix over \(\mathbb{C}\) with rank \(A_0\) = rank \(B_0\) = r.  Then, the \textbf{realization} of matrix \(A\) is \(\{H,\Pi,\Pi^V,JH,J\Pi,J\Pi^V\}\) where:
\begin{enumerate}
    \item \(H+JH\) is a real vector space with dimension \((2n-r)\).
    \item \(\Pi^V=\{\alpha_1^v,\alpha_2^v,\dots,\alpha_n^v\}\) is a set of \(n\) independent elements in \(H\).
    \item \(\Pi=\{\alpha_1,\alpha_2,\dots,\alpha_n\}\) is a set of \(n\) independent elements in the dual space \(H^*\) of \(H\).
    \item \(J\Pi^V=\{J\alpha_1^v,J\alpha_2^v,\dots,J\alpha_n^v\}\) is a set of \(n\) independent elements in \(JH\).
    \item \(J\Pi=\{J\alpha_1,J\alpha_2,\dots,J\alpha_n\}\) is a set of \(n\) independent elements in the dual space \(JH^*\) of \(JH\).
    \item The dual contraction between \(H^*\) and \(JH\) is such that \((JB_0)_{ij}= \langle a_j,Ja_i^v\rangle=J\langle a_j,a_i^v\rangle\).
    \item The dual contraction between \(JH^*\) and \(H\) is such that \((JB_0)_{ij}= \langle Ja_j,a_i^v\rangle=J\langle a_j,a_i^v\rangle\).
    \item The dual contraction between \(JH^*\) and \(JH\) is such that \(-(A_0)_{ij}= \langle Ja_j,Ja_i^v\rangle\).
\end{enumerate}
\end{definition}
Note that \(\sqrt{-1}x_i,\sqrt{-1}Jx_i\) for \(x_i=\alpha_i, \alpha_i^v\). Also, note that \(\dim H=\dim JH=2n-r\geq n\). It follows that \(\Pi^V\) is a basis in \(H\) and \(J\Pi^V\) is a basis in \(JH\) if and only if rank \(A=n\), in other words, \(A\) is a nonsingular matrix.

Any square matrix can be constructed with the following realization. A matrix \(A\) of order \(n \times n\) with rank \(A=r\) has an \(r \times r\) submatrix denoted by \(A(r)\) where \(A(r)\) is nonsingular. By permuting rows and columns, \(A\) can be rearranged so that \(A(r)\) is in the upper-left corner. Assume this has been done. Thus, the form of the matrix \(A\) is given by:
\begin{equation}
    A=\begin{pmatrix}
    A(r) & B\\ C & D
    \end{pmatrix}
\end{equation}
Next, this matrix is extended to a matrix \( E\) of order \((2n - r)\times(2n - r)\) given by:
\begin{equation}
    E = \begin{pmatrix}
    A(r) & B & 0 \\
    C & D & I_{n-r} \\
    0 & I_{n-r} & 0
\end{pmatrix}
\end{equation}
where \( I_{n-r} \) is the identity matrix of order \( (n - r) \times (n - r) \). It can be shown that the matrix \( E \) is nonsingular:
\begin{equation}
    \text{det} \, E = \pm \text{det} \, A(r) \neq 0
\end{equation}

This shows that the rows of \(E\) are linearly independent vectors of dimension \( (2n - r) \) in a vector space of dimension \( (2n - r) \). Using the matrix \( E \), the realization \( \{ H, \Pi, \Pi^V,\\ JH, J\Pi, J\Pi^V \} \) can be defined. For \( H \), take \( \mathbb{C}^{2n - r} \). The elements of \( H \) are rows of complex numbers of dimension \( (2n - r) \). The rows of matrix \( E \) provide a basis for \( H \). These rows are linearly independent and their number equals the dimension of \( H \). The system \( \Pi^V = \{ \alpha_1^v, \ldots, \alpha_n^v \} \subset H \) is defined as the set of the first \( n \) rows of \( E \). Then, for \( JH \), take \( J \mathbb{C}^{2n - r} \). The elements of \( JH \) are rows of complex numbers of dimension \( (2n - r) \). The rows of matrix \( E \) provide a basis for \( JH \). These rows are linearly independent and their number equals the dimension of \( JH \). The system \( J\Pi^V = \{ J\alpha_1^v, \ldots, J\alpha_n^v \} \subset JH \) is defined as the set of the first \( n \) rows of \( E \):
\begin{equation}
    \alpha_1^v = \{(A_0)_{11}, (A_0)_{12}, \ldots, (A_0)_{1n}, 0, 0, \ldots, 0\}
\end{equation}
\begin{equation}
    \alpha_2^v = \{(A_0)_{21}, (A_0)_{22}, \ldots, (A_0)_{2n}, 0, 0, \ldots, 0\}    
\end{equation}
\begin{equation}
    \vdots    
\end{equation}
\begin{equation}
    \alpha_r^v = \{(A_0)_{r1}, (A_0)_{r2}, \ldots, (A_0)_{rn}, 0, 0, \ldots, 0\}
\end{equation}
\begin{equation}
    \alpha_{r+1}^v = \{(A_0)_{r+1,1}, (A_0)_{r+1,2}, \ldots, (A_0)_{r+1,n}, 1, 0, \ldots, 0\} 
\end{equation}
\begin{equation}
    \vdots    
\end{equation}
\begin{equation}
    \alpha_n^v = \{(A_0)_{n1}, (A_0)_{n2}, \ldots, (A_0)_{nn}, 0, 0, \ldots, 1\}
\end{equation}
\begin{equation}
    J\alpha_1^v = \{(J B_0)_{11}, (J B_0)_{12}, \ldots, (J B_0)_{1n}, 0, 0, \ldots, 0\}
\end{equation}
\begin{equation}
    J\alpha_2^v = \{(J B_0)_{21}, (J B_0)_{22}, \ldots, (J B_0)_{2n}, 0, 0, \ldots, 0\}    
\end{equation}
\begin{equation}
    \vdots    
\end{equation}
\begin{equation}
    J\alpha_r^v = \{(J B_0)_{r1}, (J B_0)_{r2}, \ldots, (J B_0)_{rn}, 0, 0, \ldots, 0\}
\end{equation}
\begin{equation}
    J\alpha_{r+1}^v = \{(J B_0)_{r+1,1}, (J B_0)_{r+1,2}, \ldots, (J B_0)_{r+1,n}, J, 0, \ldots, 0\} 
\end{equation}
\begin{equation}
    \vdots    
\end{equation}
\begin{equation}
    J\alpha_n^v = \{(J B_0)_{n1}, (J B_0)_{n2}, \ldots, (J B_0)_{nn}, 0, 0, \ldots, J\}
\end{equation}

The space \( H^* \), the dual of \( H \), is a space of dimension \( (2n - r) \) of linear functionals on \( H \). The system \( \Pi = \{ \alpha_1, \alpha_2, \ldots, \alpha_n \} \) of linear functionals on \( H \) must be defined such that:
\begin{equation}
\langle \alpha_i, \alpha_j^v \rangle = (A_0)_{ji} \quad (i,j = 1, \ldots, n)
\end{equation}
\begin{equation}
    \langle J\alpha_i, \alpha_j^v \rangle = (J B_0)_{ji} \quad (i,j = 1, \ldots, n)
\end{equation}
\begin{equation}
    \langle \alpha_i, J\alpha_j^v \rangle = (J B_0)_{ji} \quad (i,j = 1, \ldots, n)
\end{equation}
\begin{equation}
\langle J\alpha_i, J\alpha_j^v \rangle = - (A_0)_{ji} \quad (i,j = 1, \ldots, n)
\end{equation}

For \( \alpha_i \), the linear functional is taken to assign each vector \( v = (v_1, v_2, \ldots, v_{2n - r}) \) in \( H \) the \( i \)-th component of this vector:
\begin{equation}
    \langle \alpha_i, v \rangle = v_i \quad (i = 1, \ldots, n) 
\end{equation}
\begin{equation}
    \langle J\alpha_i, v \rangle = J v_i \quad (i = 1, \ldots, n)
\end{equation}

Applying \( \alpha_i \) to the vector \( \alpha_j^v \) yields the \( i \)-th component of \( \alpha_j^v \). This exactly corresponds to the \( (j,i) \)-th element of matrix \( A_{ji} \) of matrix \( A \). This completes the construction of the realization of the matrix \( A \) of order \( n \times n \) with rank \( r \).

\subsection{Serre's Construction Over Quaternion}
\label{sec2.3}
From subsections 2.1 and 2.2, the properties needed for the construction of the universal Kac-Moody quaternion Lie algebra are obtained. Then, a realization of the generalized Cartan matrix \(A\) will be used. This realization will be denoted as \(\{H,\Pi,\Pi^V,JH,J\Pi,J\Pi^V\}\).

The starting point of the construction is with the set \((\Hat{X},J\Hat{X})\) given by
\begin{equation}
    (\Hat{X},J\Hat{X})=\Hat{X}+J\Hat{X}=\{\Hat{e}_i,J\Hat{e}_i, \Hat{f}_i, J\Hat{f}_i\}_{i=1}^n\cup\Hat{H}\cup J\Hat{H}
\end{equation}
where \( \hat{H} = \mathbb{C}^{2n-r} \), \( J\hat{H} = J\mathbb{C}^{2n-r} \), and \( \{ \hat{e}_i, J\hat{e}_i, \hat{f}_i, J\hat{f}_i \}_{i=1}^n \) is an arbitrary set with \( 4n \) elements over \( \mathbb{C} \). Based on subsection 2.1, the free quaternion Lie algebra \( L(\hat{X}, J\hat{X}) \) generated by the basis \( \{ \hat{X}, J\hat{X} \} \) has been constructed. Furthermore, \( L(\hat{X}, J\hat{X}) \) can be written as
\begin{equation}
    L(\hat{X}, J\hat{X}) = L_0(\hat{X}, J\hat{X}) + J L_0(\hat{X}, J\hat{X}),
\end{equation}
where \( L(\hat{X}, J\hat{X}) \) is the quaternionification of the complex Lie algebra \( L_0(\hat{X}, J\hat{X}) \). More specifically,
\begin{equation}
    L(\hat{X}, J\hat{X}) = L_r(\hat{X}, J\hat{X}) + \sqrt{-1} L_r(\hat{X}, J\hat{X}) + J L_r(\hat{X}, J\hat{X}) + J (\sqrt{-1} L_r(\hat{X}, J\hat{X})),
\end{equation}
where \( L_r(\hat{X}, J\hat{X}) \) is the real form of \( L_0(\hat{X}, J\hat{X}) \). Next, let us denote this quaternion Lie algebra by
\begin{equation}
    \hat{g}(A) \equiv L(\hat{X}, J\hat{X}).
\end{equation}
The quaternion Lie algebra \( \hat{g}(A) \) generated by the basis \( \{ \hat{X}, J\hat{X} \} \) has no relations among its elements. Therefore, we have the equivalence
\begin{equation}
    \hat{g}(A) = \hat{g}_0(A) + J\hat{g}_0(A),
\end{equation}
where \( \hat{g}(A) \) is the quaternionification of the complex Lie algebra \( \hat{g}_0(A) \). More specifically,
\begin{equation}
    \hat{g}(A) = \hat{g}_r(A) + \sqrt{-1} \hat{g}_r(A) + J \hat{g}_r(A) + J (\sqrt{-1} \hat{g}_r(A)),
\end{equation}
where \( \hat{g}_r(A) \) is the real form of \( \hat{g}_0(A) \). Elements of \( \hat{g}(A) \) include all generators and their linear combinations. Additionally, it includes all types of commutator multiples between generators, for example:
\begin{equation}
    [J\hat{h}, \hat{e}_i] = J\hat{h} \otimes \hat{e}_i - \hat{e}_i \otimes J\hat{h},
\end{equation}
\begin{equation}
    [\hat{e}_i, \hat{f}_i] = \hat{e}_i \otimes \hat{f}_i - \hat{f}_i \otimes \hat{e}_i,
\end{equation}
\begin{equation}
    [\hat{e}_i, J\hat{f}_i] = \hat{e}_i \otimes J\hat{f}_i - J\hat{f}_i \otimes \hat{e}_i,
\end{equation}
\begin{equation}
    [\hat{e}_i, [\hat{h}, \hat{e}_j]] = \hat{e}_i \otimes [\hat{h}, \hat{e}_j] - [\hat{h}, \hat{e}_j] \otimes \hat{e}_i = \hat{e}_i \otimes \hat{h} \otimes \hat{e}_j - \hat{e}_i \otimes \hat{e}_j \otimes \hat{h} - \hat{h} \otimes \hat{e}_j \otimes \hat{e}_i + \hat{e}_j \otimes \hat{h} \otimes \hat{e}_i,
\end{equation}
and linear combinations of such elements. Subsequently, the relationship of \( \hat{g}(A) \) will be examined through the definition of the ideal \( \hat{I} \) in \( \hat{g}(A) \). Note that in \( \hat{g}(A) \), the subset \( \hat{Y}_0 \) consists of elements:
\begin{equation}
    [\hat{h}, \hat{h}'], \quad [\hat{h}, \hat{e}_i] - \langle \hat{\alpha}_i, \hat{h} \rangle \hat{e}_i, \quad [\hat{h}, \hat{f}_i] + \langle \hat{\alpha}_i, \hat{h} \rangle \hat{f}_i, \quad [\hat{e}_i, \hat{f}_i] - \delta_{ij} \hat{\alpha}_i \hat{v}.
\end{equation}
Let \( \hat{I}_0 \) be the ideal of the complex Lie algebra in \( \hat{g}_0(A) \) generated by the subset \( \hat{Y}_0 \). This shows that \( \hat{I}_0 \) is the intersection of all ideals in \( \hat{g}_0(A) \) that contain \( \hat{Y}_0 \). Next, consider in \( \hat{g}(A) \), the subset \( \hat{Y}^* \) consisting of elements:
\begin{equation}
    [\hat{h}, J\hat{h}'], \quad [J\hat{h}, J\hat{h}'] \text{ (where } \hat{h}, \hat{h}' \in \hat{H}, J\hat{h}, J\hat{h}' \in J\hat{H}),
\end{equation}
\begin{equation}
    [J\hat{h}, \hat{e}_i] - \langle \hat{\alpha}_i, J\hat{h} \rangle \hat{e}_i, \quad [\hat{h}, J\hat{e}_i] - \langle \hat{\alpha}_i, \hat{h} \rangle J\hat{e}_i, \quad [J\hat{h},  J\hat{e}_i] + \langle \hat{\alpha}_i, \hat{h} \rangle \hat{e}_i,
\end{equation}
\begin{equation}
    [J\hat{h}, \hat{f}_i] + \langle \hat{\alpha}_i, J\hat{h} \rangle \hat{f}_i, \quad [\hat{h}, J\hat{f}_i] + \langle \hat{\alpha}_i, \hat{h} \rangle J\hat{f}_i, \quad [J\hat{h}, J\hat{f}_i] - \langle \hat{\alpha}_i, \hat{h} \rangle \hat{f}_i,
\end{equation}
\begin{equation}
    [J\hat{e}_i, \hat{f}_i] - \delta_{ij} J \hat{\alpha}_i^v, \quad  [\hat{e}_i, J\hat{f}_i] - \delta_{ij} J \hat{\alpha}_i^v, \quad [J\hat{e}_i, J\hat{f}_i] + \delta_{ij} \hat{\alpha}_i^v \quad (i, j = 1, \ldots, n).
\end{equation}
Let \(\hat{I}^*\) be an ideal of the real Lie algebra \(\hat{g}(A)\) generated by the subset \(\hat{Y}^*\). This shows that \(\hat{I}^*\) is the intersection of all ideals in \(\hat{g}(A)\) that contain \(\hat{Y}^*\).

Let \(\hat{I} = \hat{I}_0 + \hat{I}^*\), where \(\hat{I}\) is generated by the subsets \(\hat{Y}_0\) and \(\hat{Y}^*\). Then,
\begin{equation}
    \hat{I} \subset L_0 + JL_r, \quad \hat{I} \cap \hat{g}_0(A) = \hat{I}_0.
\end{equation}
Note that elements such as \([ \sqrt{-1} x_i, J y_i ]\) and \([ \sqrt{-1} J x_i, J y_i ]\) for \(x_i, y_i = h_i, e_i, f_i\) do not necessarily have to be contained in \(\hat{I}\). From \(\hat{g}(A)\) and the ideal \(\hat{I}\), the quaternion Lie algebra quotient \(\tilde{g}(A)\) is defined by
\begin{equation}
    \tilde{g}(A) = \hat{g}(A) / \hat{I}.
\end{equation}
Thus, \(\tilde{g}(A)\) is the quaternionification of \(\tilde{g}(A)_0\). The canonical projection from \(\hat{g}(A)\) to \(\tilde{g}(A)\) given by \(\Psi\) is:
\begin{equation}
    \Psi: \hat{x} \in \hat{g}(A) \rightarrow x \equiv \Psi(\hat{x}) \in \tilde{g}(A).
\end{equation}
Since \(\Psi\) is a quaternion Lie algebra homomorphism, we obtain the commutation relations for \(\tilde{g}(A)\):
\begin{equation}
    [\hat{h}, \hat{h}'] = 0, \quad [\hat{h}, \hat{e}_i] = \langle \hat{\alpha}_i, \hat{h} \rangle \hat{e}_i, \quad [\hat{e}_i, \hat{f}_j] = \delta_{ij} \hat{\alpha}_i v, \quad [\hat{h}, \hat{f}_i] = -\langle \hat{\alpha}_i, \hat{h} \rangle \hat{f}_i,
\end{equation}
and
\begin{equation}
    [\hat{h}, J\hat{h}'] = 0, \quad [J\hat{h}, J\hat{h}'] = 0 \quad (\hat{h}, \hat{h}' \in \hat{H}, J\hat{h}, J\hat{h}' \in J\hat{H}),
\end{equation}
\begin{equation}
    [J\hat{h}, \hat{e}_i] = \langle \hat{\alpha}_i, J\hat{h} \rangle \hat{e}_i, \quad [\hat{h}, J\hat{e}_i] = \langle \hat{\alpha}_i, \hat{h} \rangle J\hat{e}_i, \quad [J\hat{h}, J\hat{e}_i] = -\langle \hat{\alpha}_i, \hat{h} \rangle \hat{e}_i,
\end{equation}
\begin{equation}
    [J\hat{h}, \hat{f}_i] = -\langle \hat{\alpha}_i, J\hat{h} \rangle \hat{f}_i, \quad [\hat{h}, J\hat{f}_i] = -\langle \hat{\alpha}_i, \hat{h} \rangle J\hat{f}_i, \quad [J\hat{h}, J\hat{f}_i] = \langle \hat{\alpha}_i, \hat{h} \rangle \hat{f}_i,
\end{equation}
\begin{equation}
    [J\hat{e}_i, \hat{f}_j] = \delta_{ij} J \hat{\alpha}_i^v, \quad [\hat{e}_i, J\hat{f}_j] = \delta_{ij} J \hat{\alpha}_i^v, \quad [J\hat{e}_i, J\hat{f}_j] = -\delta_{ij} \hat{\alpha}_i^v \quad (i, j = 1, \ldots, n)
\end{equation}
Note that there are no carets \((\quad\Hat{ }\quad)\) on \(\alpha_i\) and \(\alpha_i^v\) in (75)-(79).

Next, it will be demonstrated that the quaternion Lie algebra $\tilde{g}(A)$ is non-trivial. This can be shown by constructing a representation $\hat{\psi}$ of $\tilde{g}(A)$. To investigate the structure of the quaternion Lie algebra $\tilde{g}(A)$, the set of generators $\{X, JX\}$ will be used. Each element of the generator set $\{X, JX\}$, namely $\hat{h}$, $J\hat{h}$, $\hat{e}_i$, $J\hat{e}_i$, $\hat{f}_i$, and $J\hat{f}_i$, will be defined as the corresponding linear operators $\hat{h} \cdot$, $J\hat{h} \cdot$, $\hat{e}_i \cdot$, $J\hat{e}_i \cdot$, $\hat{f}_i \cdot$, and $J\hat{f}_i \cdot$, which act on the quaternion tensor algebra $T(V)$, where $(V, J)$ is an $n$-dimensional quaternion module with $\{v_1, v_2, \ldots, v_n\}$. This defines a mapping

\begin{equation}
    \phi: (\hat{X}, J\hat{X}) \to gl(T(V))
\end{equation}

Based on (32), this mapping has a unique extension to the representation $\hat{\psi}$ (over $\mathbb{R}$) of the algebra $\tilde{g}(A)$ on the space $T(V)$. For simplicity in notation, elements of the quaternion tensor algebra $T(V)$ will be denoted by

\begin{equation}
    v_{i_1} v_{i_2} \cdots v_{i_k} \equiv v_{i_1} \otimes v_{i_2} \otimes \cdots \otimes v_{i_k}
\end{equation}
for $k = 1, 2, \ldots$. For $k = 0$, the base element 1 $\in \mathbb{H}$ is taken.

Let \(T_0(V)\) be the complex tensor algebra associated with the quaternion module \((\mathcal{V},J)\) of dimension \(n\), and let \(\hat{X} = \{ \hat{e}_i, \hat{f}_i \}_{i=1}^n \cup \hat{H}\) be the set of generators. Then the action of \(\hat{X}\) on \(T_0(V)\), for any \(\hat{\lambda} \in \hat{H}^*\), is defined by:

\begin{equation}
    \hat{h} \bullet 1 := \langle \lambda, \hat{h} \rangle 1
\end{equation}
\begin{equation}
    \hat{h} \bullet v_{j_1} \cdots v_{j_k} := -\langle \hat{\alpha}_{j_i} + \cdots + \hat{\alpha}_{j_k}, \hat{h} \rangle v_{j_1} \cdots v_{j_k} + \langle \hat{\lambda}, \hat{h} \rangle v_{j_1} \cdots v_{j_k}
\end{equation}
\begin{equation}
    \hat{f}_i \bullet 1 := v_i
\end{equation}
\begin{equation}
    \hat{f}_i \bullet v_{j_1} \cdots v_{j_k} := v_i v_{j_1} \cdots v_{j_k}
\end{equation}
\begin{equation}
    \hat{e}_i \bullet 1 := 0
\end{equation}
\begin{equation}
    \hat{e}_i \bullet v_{j_1} \cdots v_{j_k} := v_i \hat{e}_i \cdot (v_{j_1} \cdots v_{j_k}) + \delta_{i j_1} \hat{\alpha}_i v \cdot (v_{j_1} \cdots v_{j_k})
\end{equation}
The action of $\hat{X}$ on $T_0(V)$ can be extended to a Lie algebra homomorphism $\hat{\psi}_0$ (over $\mathbb{C}$) from $\hat{g}_0(A)$ to the linear quaternion Lie algebra $gl(T_0(V))$. Furthermore, the action of $\hat{X}$ on $T_0(V)$ extends to the action $(\hat{X}, J\hat{X})$ on $T(V)$, where $T(V)$ is the quaternionification of $T_0(V)$. The definition of the action $(\hat{X}, J\hat{X})$ on $T(V)$ follows.
\begin{definition}
    Let $(\hat{X}, J\hat{X})$ be the set of generators (57), and let $T(V)$ be the real tensor algebra associated with the quaternion module $(V, J)$ of dimension $n$. Then the action of $(\hat{X}, J\hat{X})$ on $T(V)$, for any $\hat{\lambda} \in \hat{H}^*$ and $J\hat{\lambda} \in J\hat{H}^*$, is defined by:
    \begin{equation}
         \Hat{h}\bullet J:=\langle J\lambda,\Hat{h}\rangle 1
    \end{equation}
    \begin{equation}
        \Hat{f_i}\bullet J:=Jv_i
    \end{equation}
    \begin{equation}
        \Hat{e_i}\bullet J:=0
    \end{equation}
    \begin{equation}
        J\Hat{h}\bullet 1:=\langle\lambda,J\Hat{h}\rangle1,  J\Hat{h}\bullet J:=-\langle\lambda,\Hat{h}\rangle 1
    \end{equation}
    \begin{equation}
        \Hat{Jh}\bullet {v_j}_1\dots {v_j}_k:=-\langle \Hat{\alpha}_{j_1}+\dots+\Hat{\alpha}_{j_k},J\Hat{h}\rangle {v_j}_1\dots {v_j}_k+\langle\lambda,J\Hat{h}\rangle {v_j}_1\dots {v_j}_k
    \end{equation}
    \begin{equation}
        J\Hat{f_i}\bullet 1:=Jv_i,  J\Hat{f_i}\bullet J:=-v_i
    \end{equation}
    \begin{equation}
        J\Hat{f_i}\bullet {v_j}_1\dots {v_j}_k:=Jv_i{v_j}_1\dots {v_j}_k
    \end{equation}
    \begin{equation}
        J\Hat{e_i}\bullet 1:=0,  J\Hat{e_i}\bullet J:=0
    \end{equation}
    \begin{equation}
        J\Hat{e_i}\bullet {v_j}_1\dots {v_j}_k:={v_j}_1J\Hat{e_i}\bullet({v_j}_2\dots {v_j}_k)+{\delta_{ij}}_1J\Hat{\alpha}_i^v\bullet ({v_j}_2\dots {v_j}_k)
    \end{equation}
    \begin{equation}
        \Hat{h}\bullet J{v_j}_1\dots {v_j}_k:=-\langle \Hat{\alpha}_{j_1}+\dots+\Hat{\alpha}_{j_k},\Hat{h}\rangle J{v_j}_1\dots {v_j}_k+\langle\lambda,\Hat{h}\rangle J{v_j}_1\dots {v_j}_k
    \end{equation}
    \begin{equation}
        \Hat{f_i}\bullet J{v_j}_1\dots {v_j}_k:=v_iJ{v_j}_1\dots {v_j}_k
    \end{equation}
    \begin{equation}
        \Hat{e_i}\bullet J{v_j}_1\dots {v_j}_k:={v_j}_1\Hat{e_i}\bullet(J{v_j}_2\dots {v_j}_k)+{\delta_{ij}}_1\Hat{\alpha}_i^v\bullet (J{v_j}_2\dots {v_j}_k)
    \end{equation}
    \begin{equation}
        J\Hat{h}\bullet J{v_j}_1\dots {v_j}_k:=\langle \Hat{\alpha}_{j_1}+\dots+\Hat{\alpha}_{j_k},\Hat{h}\rangle {v_j}_1\dots {v_j}_k-\langle\lambda,\Hat{h}\rangle {v_j}_1\dots {v_j}_k
    \end{equation}
    \begin{equation}
        J\Hat{f_i}\bullet J{v_j}_1\dots {v_j}_k:=-v_i{v_j}_1\dots {v_j}_k
    \end{equation}
    \begin{equation}
        J\Hat{e_i}\bullet J{v_j}_1\dots {v_j}_k:=-{v_j}_1\Hat{e_i}\bullet({v_j}_2\dots {v_j}_k)-{\delta_{ij}}_1\Hat{\alpha}_i^v\bullet ({v_j}_2\dots {v_j}_k)
    \end{equation}
\end{definition}
The action of \((\hat{X}, J\hat{X})\) on \(T(V)\) can be extended to a quaternion Lie algebra homomorphism \(\hat{\psi}\) (over \(\mathbb{R}\)) from \(\hat{g}(A)\) to the quaternion linear Lie algebra \(gl(T(V))\):
\begin{equation}
    \hat{\psi}: \hat{x} \in \hat{g}(A) \rightarrow \hat{\psi}(\hat{x}) \in gl(T(V)).
\end{equation}
This is given by
\begin{equation}
    \hat{\psi}([ \hat{x}, \hat{y}]) = \hat{x} \cdot \hat{y} - \hat{y} \cdot \hat{x}.
\end{equation}
Next, it will be shown that a representation \(\tilde{g}(A)\) can be obtained that satisfies the commutation relations (75)-(79). To achieve these commutation relations, the following theorem is needed.
\begin{theorem}
    Let \(\psi\) be a homomorphism of quaternion Lie algebra from a quaternion Lie algebra \(K\) to a quaternion Lie algebra \(L\) and Let \(I\) be an ideal in \(K\) contained in the kernel \(\psi\) i.e. \(I\subset\) ker \( \psi\). Then there exist a unique homomorphism of quaternion Lie algebra \(\phi\) from \(K/I\) to \(L\) such that
    \begin{equation}\label{eq1}
    \psi=\phi\circ\Psi
    \end{equation}
    where \(\Psi:K\to K/I\) is canonical projection.
\end{theorem}
\begin{proof}
    Define \(\phi:K/I\to L\) by \(\phi(k+I)=\psi(k), \forall k\in K\). Since \(\psi\) is a homomorphism of quaternion Lie algebra, \(\phi\) is well defined. Suppose for some \(k_1,k_2\in K\) that \(k_1+I=k_2+I\). Since \(I\) is an ideal, it contains the zero \(0_K\) of \(K\). Then for some \(i\in I\), we have that \(k_1+0_K=k_1=k_2+i\). Then
    \begin{equation}\begin{split} \phi(k_1+I)&=\psi(k_1)\\&=\psi(k_2+i)\\&=\psi(r_2)+\psi(i)\\&=\psi(k_2)+0_K\\&=\psi(k_2)\\&=\phi(k_2+I)
    \end{split}
    \end{equation}
    so \(\phi\) is seen to be well defined. Next, for all \(k_1, k_2\in K\).
    \begin{equation}
        \begin{split}
            \phi([k_1+I,k_2+I])&=\phi((k_1+I)(k_2+I)-(k_2+I)(k_1+I))\\&=\phi(k_1+I)\phi(k_2+I)-\phi(k_2+I)\phi(k_1+I)\\&=\psi(k_1)\psi(k_2)-\psi(k_2)\psi(k_1)\\&=[\psi(k_1),\psi(k_2)]\\&=[\phi(k_1+I),\phi(k_2+I)]
        \end{split}
    \end{equation}
    Therefore \(\phi\) is homomorphism Lie algebra. Furthermore, we have \(\phi(\sigma (k+I))=\sigma\phi(k+I)\) and \(\phi(\tau (k+I))=\tau\phi(k+I)\), for all . It remains to demonstrate uniqueness. Suppose there were another homomorphism \(\phi':K/I\to L\) with \(\psi=\phi'\circ \Psi\). Then we would require that, for all \(k\in K\). \(\phi'(k+I)=\psi(k)=\phi(k+I)\). That is \(\phi'=\phi\) and so \(\phi\) is unique.
\end{proof}
To apply Theorem 2.1 to the representation \( \hat{\psi} \) of \( \hat{g}(A) \) on \( T(V) \), it is first shown that the ideal \( \hat{I} \) is contained in the kernel \( \ker \hat{\psi} \).

First, consider the commutator \( [\hat{h}, \hat{h}'] \). From the definition of \( \hat{h} \), particularly noting that \( \hat{h} \cdot \) and \( \hat{h}' \cdot \) act diagonally on the basis, we have
\begin{equation}
    \hat{\psi}([\hat{h}, \hat{h}']) = \hat{h} \cdot \hat{h}' \cdot - \hat{h}' \cdot \hat{h} \cdot = 0.
\end{equation}
Thus,
\begin{equation}
    [\hat{h}, \hat{h}'] \in \ker \hat{\psi}.
\end{equation}
Similarly, it can be shown that
\begin{equation}
    [\hat{h}, J\hat{h}'] \in \ker \hat{\psi}
\end{equation}
\begin{equation}
    [J\hat{h}, J\hat{h}'] \in \ker \hat{\psi}
\end{equation}
Next, consider the example \( [\hat{h}, \hat{f}_j] \). We have
\begin{equation}
    \hat{\psi}([\hat{h}, \hat{f}_j]) = \hat{h} \cdot \hat{f}_j - \hat{f}_j \cdot \hat{h}
\end{equation}
Next, it will be shown that the left-hand side is equal to \( \hat{\psi}(-\langle \hat{\alpha}_j, \hat{h} \rangle \hat{f}_j) \) by applying the operator (112) to \( v_{i_1} \cdots v_{i_k} \in T(V) \). Using (83) and (85), we obtain:
\begin{equation}\begin{split}
    (\hat{h} \cdot \hat{f}_j -& \hat{f}_j \cdot \hat{h}) v_{i_1} \cdots v_{i_k} \\&= \hat{h} \cdot v_j v_{i_1} \cdots v_{i_k} + \hat{f}_j \cdot \left\{\langle \hat{\alpha}_{i_1} + \cdots + \hat{\alpha}_{i_k}, \hat{h} \rangle v_{i_1} \cdots v_{i_k} - \langle \hat{\lambda}, \hat{h} \rangle v_j v_{i_1} \cdots v_{i_k} \right\}\\&= -\langle \hat{\alpha}_j + \hat{\alpha}_{i_1} + \cdots + \hat{\alpha}_{i_k}, \hat{h} \rangle v_j v_{i_1} \cdots v_{i_k} + \langle \hat{\lambda}, \hat{h} \rangle v_j v_{i_1} \cdots v_{i_k}\\& + \langle \hat{\alpha}_{i_1} + \cdots + \hat{\alpha}_{i_k}, \hat{h} \rangle v_j v_{i_1} \cdots v_{i_k} - \langle \hat{\lambda}, \hat{h} \rangle v_j v_{i_1} \cdots v_{i_k}\\& = -\langle \hat{\alpha}_j, \hat{h} \rangle \hat{f}_j \cdot v_{i_1} \cdots v_{i_k}\\& = \hat{\psi}(-\langle \hat{\alpha}_j, \hat{h} \rangle \hat{f}_j) v_{i_1} \cdots v_{i_k}
    \end{split}
\end{equation}
From (112) and (113), we obtain:
\begin{equation}
    \hat{\psi}([\hat{h}, \hat{f}_j] + \langle \hat{\alpha}_j, \hat{h} \rangle \hat{f}_j) = 0
\end{equation}
Therefore,
\begin{equation}
    ([\hat{h}, \hat{f}_j] + \langle \hat{\alpha}_j, \hat{h} \rangle \hat{f}_j) \in \ker \hat{\psi}
\end{equation}
Similarly, we obtain:
\begin{align*}
    ([J\hat{h}, \hat{f}_j] + \langle \hat{\alpha}_j, J\hat{h} \rangle \hat{f}_j) &\in \ker \hat{\psi}             & ([J\hat{h}, J\hat{e}_i] - \langle \hat{\alpha}_i, \hat{h} \rangle \hat{e}_j) &\in \ker \hat{\psi}\\ ([\hat{h}, J\hat{f}_j] + \langle \hat{\alpha}_j, \hat{h} \rangle J\hat{f}_j) &\in \ker \hat{\psi} & ([\hat{e}_i, \hat{f}_j] - \delta_{ij} \hat{\alpha}_i v) &\in \ker \hat{\psi}\\([J\hat{h}, J\hat{f}_j] - \langle \hat{\alpha}_j, \hat{h} \rangle \hat{f}_j) &\in \ker \hat{\psi} & ([J\hat{e}_i, \hat{f}_j] - \delta_{ij} J \hat{\alpha}_i v) &\in \ker \hat{\psi}\\ ([\hat{h}, \hat{e}_i] + \langle \hat{\alpha}_i, \hat{h} \rangle \hat{e}_j) &\in \ker \hat{\psi} & ([J\hat{h}, \hat{e}_i] + \langle \hat{\alpha}_i, J\hat{h} \rangle \hat{e}_j) &\in \ker \hat{\psi}\\ ([J\hat{e}_i, J\hat{f}_j] + \delta_{ij} \hat{\alpha}_i v) &\in \ker \hat{\psi} & ([\hat{h}, J\hat{e}_i] + \langle \hat{\alpha}_i, \hat{h} \rangle J\hat{e}_j) &\in \ker \hat{\psi}\\ \left[ \hat{e}_i, J \hat{f}_j \right] - \delta_{ij} J \hat{\alpha}_i v &\in \ker \hat{\psi}
\end{align*}

Thus, it is obtained that every generator element of the ideal \( \hat{I} \) is in \( \ker \hat{\psi} \) and \( \hat{I} \subseteq \ker \hat{\psi} \). Now, we can apply Theorem 2.1 to the homomorphism \( \hat{\psi} \) in (103) with \( K = \hat{g}(A) \) and \( L = gl(T(V)) \). This concludes the representation \( \phi \) of \( \tilde{g}(A) \) in (73):
\begin{equation}
    \varphi: \tilde{g}(A) = \hat{g}(A) / \hat{I} \to gl(T(V))
\end{equation}
Using the notation \( x = \Psi(\hat{x}) \) in (74), we obtain
\begin{equation}
    \varphi(x) = \hat{\psi}(\hat{x})
\end{equation}

Now, all the required properties have been obtained to define the Universal Kac-Moody Quaternion Lie Algebra.
\begin{definition}
    Let \(\Hat{g}(A)\) be the Quaternion Lie algebra defined in (60) and \(\Hat{I}\) the ideal defined in (67)-(71). Then the quotient algebra
    \begin{equation}\label{2}
        \Tilde{g}(A)=\Hat{g}(A)/\Hat{I}
    \end{equation} 
    is by definition the \textbf{Universal Kac-Moody Quaternion Lie Algebra} corresponding to the generalized Cartan matrix \(A\).
\end{definition}
\begin{theorem}
    Let \(\Tilde{g}(A)\) be the Universal Kac-Moody Quaternion Lie algebra belonging to the generalized Cartan matrix \(A=(A_{ij})\). Then \(\Tilde{g}(A)\) has the following properties:
    \begin{enumerate}
        \item In the Universal Kac-Moody Quaternion Lie algebra \(\Tilde{g}(A)\) is genereted by 
        \begin{equation}
           \{X,JX\}=\{e_i,Je_i,f_i,Jf_i\}_{i=1}^n\cup H\cup JH 
        \end{equation}
        \item In the Universal Kac-Moody Quaternion Lie algebra \(\Tilde{g}(A)\) the following commutation realtions hold
        \begin{equation}
            [h,e_i]=\langle \alpha_i,h\rangle e_i, [Jh,e_i]=\langle \alpha_i,Jh\rangle e_i, [h,Je_i]=\langle \alpha_i,Jh\rangle e_i,
        \end{equation}
        \begin{equation}
            [Jh,Je_i]=-\langle \alpha_i,h\rangle e_i, [h,f_i]=-\langle \alpha_i,f_i\rangle,  [Jh,f_i]=-\langle \alpha_i,Jh\rangle f_i,
        \end{equation}
        \begin{equation}
            [h,Jf_i]=-\langle \alpha_i,h\rangle Jf_i, [Jh,Jf_i]=\langle \alpha_i,h\rangle f_i,
        \end{equation}
        \begin{equation}
            [e_i,f_j]=\delta_{ij} \alpha_i^v,[Je_i,f_j]=\delta_{ij} J\alpha_i^v,[e_i,Jf_j]=\delta_{ij} J\alpha_i^v,
        \end{equation}
        \begin{equation}
            [Je_i,Jf_j]=\delta_{ij} J\alpha_i^v,[h,h]=0,[h,Jh]=0,[Jh,Jh]=0 
        \end{equation}
        Special case of (120)-(124) read
        \begin{equation}
            [\alpha_i^v,e_i]=\langle \alpha_i,\alpha_i^v\rangle e_i=A_{ij}e_j, [J\alpha_i^v,e_i]=\langle \alpha_i,J\alpha_i^v\rangle e_i=A_{ij}e_j,
        \end{equation}
        \begin{equation}
            [\alpha_i^v,Je_i]=\langle \alpha_i,\alpha_i^v\rangle Je_i=A_{ij}e_j, [J\alpha_i^v,Je_i]=-\langle \alpha_i,\alpha_i^v\rangle e_i=-A_{ij}e_j
        \end{equation}
        and
        \begin{equation}
            [\alpha_i^v,f_i]=-\langle \alpha_i,\alpha_i^v\rangle f_i=-A_{ij}f_j, [J\alpha_i^v,f_i]=-\langle \alpha_i,J\alpha_i^v\rangle f_i=-A_{ij}f_j,
        \end{equation}
        \begin{equation}
            [\alpha_i^v,Jf_i]=-\langle \alpha_i,\alpha_i^v\rangle Jf_i=-A_{ij}f_j, [J\alpha_i^v,Jf_i]=\langle \alpha_i,\alpha_i^v\rangle f_i=A_{ij}f_j
        \end{equation}
        \item Let \(\Tilde{N}^-\) be the \(\mathbb{C}\)-module generated by \(\{f_1,\dots,f_n,Jf_1,\dots,Jf_n\}\), \(\Tilde{N}^+\) be the \(\mathbb{C}\)-module generated by \(\{e_1,\dots,e_n,Je_1,\dots,Je_n\}\), and Let \(K\) generated by \(\mathbb{H}\otimes_{\mathbb{C}} H=H\oplus JH\) . Then, \(\Tilde{N}^-,K, \text{and }\Tilde{N}^+\) viewed as a real Lie subalgebra of \(\Tilde{g}(A)\) give the decomposition of \(\Tilde{g}(A):\) 
        \begin{equation}
            \Tilde{g}(A)=\Tilde{N}^-\oplus K\oplus\Tilde{N}^+
        \end{equation}
        \item Considering \(\Tilde{g}(A)\) as an ad \(H_r\) module, where \(H_r\) is a maximal commutative subalgebra of \(\Tilde{g}(A)\) we have the root decomposition  
        \begin{equation}
            \Tilde{g}(A)=\left(\bigoplus_{\alpha\ne0, \alpha\in Q^+}\Tilde{g}_{-\alpha}\right)\oplus \Tilde{g}_0\oplus \left(\bigoplus_{\alpha\ne0, \alpha\in Q^+}\Tilde{g}_{+\alpha}\right)
        \end{equation}
        where \(\Tilde{g}_\alpha:=\{x\in \Tilde{g}(A)|\forall h\in H_r:\text{ad } h(x)=\langle \alpha,h\rangle x\}, H_r\subset \Tilde{g}_0=K\). Furthermore, dim \(\Tilde{g}_\alpha<\infty\) and \(\Tilde{g}_\alpha\subset \Tilde{N}^\pm\) for \(\pm \alpha \in Q^+, \alpha\ne0\).
    \end{enumerate}
\end{theorem}
\begin{proof}
To prove Theorem 2.1, we first need to show that the set
\begin{equation}
    \hat{H} \cup J\hat{H} \cup \{ \hat{e}_i, \hat{f}_i, J\hat{e}_i, J\hat{f}_i \}_{i=1}^{n}
\end{equation}
satisfies the canonical projection
\begin{equation}
    \Psi: \hat{g}(A) \to \tilde{g}(A)
\end{equation}
This means we need to ensure that \(\hat{H}\), \(J\hat{H}\), and \(\{ \hat{e}_i, \hat{f}_i, J\hat{e}_i, J\hat{f}_i \}_{i=1}^{n}\) are not contained in the ideal \(\hat{I}\).

However, it can be further shown that both \(\hat{H}\), \(J\hat{H}\), and \(\{ \hat{e}_i, \hat{f}_i, J\hat{e}_i, J\hat{f}_i \}_{i=1}^{n}\) are not contained in \(\text{ker} \, \hat{\psi}\). For \(\hat{H}\) and \(J\hat{H}\), this is relatively straightforward to prove. 

Take \(\hat{h} \in \hat{H}\) and \(J\hat{h} \in J\hat{H}\). Suppose \(\hat{h}\) and \(J\hat{h}\) are elements of \(\text{ker} \, \hat{\psi}\). Then we have
\begin{equation}
    \hat{\psi}(\hat{h}) = 0
\end{equation}
and
\begin{equation}
    \hat{\psi}(J\hat{h}) = 0,
\end{equation}
From (82) and (91), we have
\begin{equation}
    \hat{\psi}(\hat{h}) \cdot 1 = \hat{h} \cdot 1 = \langle \hat{\lambda}, \hat{h} \rangle_1 = 0,
\end{equation}
\begin{equation}
    \hat{\psi}(J\hat{h}) \cdot 1 = J\hat{h} \cdot 1 = \langle \hat{\lambda}, J\hat{h} \rangle_1 = 0,
\end{equation}
where \(\hat{\lambda}\) is arbitrary in \(H^*\). However, if \(\hat{h}\) and \(J\hat{h}\) are in \(\text{ker} \, \hat{\psi}\), then we get
\begin{equation}
    \langle \hat{\lambda}, \hat{h} \rangle = \langle \hat{\lambda}, J\hat{h} \rangle = 0 \text{ for all } \hat{\lambda} \in H^*.
\end{equation}
Thus, \(\hat{h} = J\hat{h} = 0\). Therefore,
\begin{equation}
    \hat{H} = J\hat{H} = 0,
\end{equation}
which is a contradiction. This implies that \(\hat{H} \cap \text{ker} \, \hat{\psi} = \{0\}\) and \(J\hat{H} \cap \text{ker} \, \hat{\psi} = \{0\}\).

It is now known that the subspace \(\hat{H}\) of the quaternion Lie algebra \(\hat{g}(A)\) is mapped one-to-one onto \(\tilde{g}(A)\). That is,
\begin{equation}
    \Psi: \hat{H} \to H
\end{equation}
\begin{equation}
    \Psi: J\hat{H} \to JH
\end{equation}
is a bijective mapping. Thus, \(\text{dim} \, H = \text{dim} \, \hat{H}\) and \(\text{dim} \, JH = \text{dim} \, J\hat{H}\).

To prove that the generator elements \(\hat{e}_1, \ldots, \hat{e}_n, J\hat{e}_1, \ldots, J\hat{e}_n, \hat{f}_1, \ldots, \hat{f}_n, J\hat{f}_1, \ldots, J\hat{f}_n\) are not in \(\text{ker} \, \hat{\psi}\), we will use the fact that the mapping \(\hat{\psi}: \hat{g}(A) \to gl(T(V))\) is a quaternion Lie algebra homomorphism. Elements (67)-(71) in \(\hat{g}(A)\) are mapped by \(\hat{\psi}\) to the zero elements of \(gl(T(V))\) (see, for example, (115)).

Let \(\hat{h} = \hat{\alpha}_i v\) and \(J\hat{h} = J\hat{\alpha}_i v\). We have:
\begin{align*}
   [\hat{e}_i \cdot, \hat{f}_i \cdot] &= \hat{\alpha}_i v \cdot, & [J\hat{e}_i \cdot, \hat{f}_i \cdot] &= J\hat{\alpha}_i v \cdot, & [\hat{e}_i \cdot, J\hat{f}_i \cdot] &= J\hat{\alpha}_i v \cdot, & [J\hat{e}_i \cdot, J\hat{f}_i \cdot] &= -\hat{\alpha}_i v \cdot, \\ [\hat{\alpha}_i v \cdot, \hat{e}_i \cdot] &= 2 \hat{e}_i \cdot, & [\hat{\alpha}_i v \cdot, J\hat{e}_i \cdot] &= 2 J\hat{e}_i \cdot, & [J\hat{\alpha}_i v \cdot, \hat{e}_i \cdot] &= 2 J\hat{e}_i \cdot, & [J\hat{\alpha}_i v \cdot, J\hat{e}_i \cdot] &= -2 \hat{e}_i \cdot, \\ [\hat{\alpha}_i v \cdot, \hat{f}_i \cdot] &= -2 \hat{f}_i \cdot, & [\hat{\alpha}_i v \cdot, J\hat{f}_i \cdot] &= -2 J\hat{f}_i \cdot, & [J\hat{\alpha}_i v \cdot, \hat{f}_i \cdot] &= -2 J\hat{f}_i \cdot, & [J\hat{\alpha}_i v \cdot, J\hat{f}_i \cdot] &= 2 \hat{f}_i \cdot
\end{align*}
This shows that we obtain \(gl(T(V))\) from the representation \(\mathfrak{sl}(2; \mathbb{H})\). Since \(\mathfrak{sl}(2; \mathbb{H})\) is a simple quaternion Lie algebra whose representation is faithful, from the discussion (135) and (136), we have that \(\hat{\psi}(\hat{\alpha}_i v) = \hat{\alpha}_i v \cdot \neq 0 \text{ and } \hat{\psi}(J\hat{\alpha}_i v) = J\hat{\alpha}_i v \cdot \neq 0\)
Thus, we obtain a faithful representation and consequently, \(\hat{\psi}(\hat{e}_i) = \hat{e}_i \cdot \neq 0, \hat{\psi}(J\hat{e}_i) = J\hat{e}_i \cdot \neq 0, \hat{\psi}(\hat{f}_i) = \hat{f}_i \cdot \neq 0, \hat{\psi}(J\hat{f}_i) = J\hat{f}_i \cdot \neq 0.\)
It follows that \(\hat{e}_i, J\hat{e}_i, \hat{f}_i,\) and \(J\hat{f}_i\) are not in \(\text{ker} \, \hat{\psi}\).
Next, using the notation introduced in (74) and (132), we obtain
\begin{equation}
    \begin{split}
        \Psi(\hat{e}_i) = e_i, \quad \Psi(J\hat{e}_i) = Je_i, \quad \Psi(\hat{f}_i) = f_i, \\ \quad \Psi(J\hat{f}_i) = Jf_i, \Psi(\hat{h}) = h, \quad \Psi(J\hat{h}) = Jh,
    \end{split}
\end{equation}
where \(h \in H\) and \(Jh \in JH\); \(i = 1, \ldots, n\).

According to \(\Psi: \hat{H} \to H\) and \(\Psi: J\hat{H} \to JH\), the projection \(\Psi^*\) between their dual spaces is given by:
\begin{equation}
    \Psi^*: \hat{\gamma} \in \hat{H}^* \to \gamma \in H^*,
\end{equation}
\begin{equation}
    \Psi^*: J\hat{\gamma} \in J\hat{H}^* \to J\gamma \in JH^*.
\end{equation}
This is defined in such a way that the dual contraction is invariant. It follows that for every \(\hat{\gamma} \in \hat{H}^*\) and \(J\hat{\gamma} \in J\hat{H}^*\), we have:
\begin{equation}
    \langle \hat{\gamma}, \hat{h} \rangle = \langle \gamma, h \rangle, \quad \langle J\hat{\gamma}, \hat{h} \rangle = \langle J\gamma, h \rangle, \quad \langle \hat{\gamma}, J\hat{h} \rangle = \langle \gamma, J h \rangle, \quad \langle J\hat{\gamma}, J\hat{h} \rangle = \langle J\gamma, J h \rangle
\end{equation}
where
\begin{equation}
    \Psi(\hat{h}) = h, \quad \Psi(J \hat{h}) = J h
\end{equation}
and
\begin{equation}
    \Psi^*(\hat{\gamma}) = \gamma, \quad \Psi^*(J \hat{\gamma}) = J \gamma
\end{equation}
Thus, we obtain the set \(\{H, \Pi, \Pi^V, JH, J \Pi, J \Pi^V\}\), where
\begin{equation}
    \langle \alpha_i, \alpha_j^v \rangle = A_{ji}, \quad \langle J \alpha_i, \alpha_j^v \rangle = J A_{ji}, \quad \langle \alpha_i, J \alpha_j^v \rangle = J A_{ji}, \quad \langle J \alpha_i, J \alpha_j^v \rangle = -A_{ji}
\end{equation}
This represents the realization of the Cartan matrix generalization \(A_{ij}\).

Next, the proofs of (a)-(d) in Theorem 2.1 will be shown. Part (a) has been proven since it has been demonstrated that the generating set
\begin{equation}
    \hat{H} \cup J \hat{H} \cup \{ \hat{e}_i, \hat{f}_i, J \hat{e}_i, J \hat{f}_i \}_{i=1}^n
\end{equation}
satisfies the canonical projection \(\Psi\) in the previous section. 

Part (b) simply demonstrates the result of constructing the quaternion Lie algebra quotient \(\hat{g}(A) / \hat{I}\). To prove part (c), we refer to Theorem 3.7 in (Kori, 2023). From this theorem, it follows that \(\tilde{g}(A) = \tilde{N}^{-} + K + \tilde{N}^{+}\). Assume \(u = \tilde{N}^{-} + k + \tilde{N}^{+} = 0\), where \(\tilde{N}^{\pm} \in \tilde{N}^{\pm}\) and \(k \in K\). Since \(\tilde{N}^{-}\) is a \(\mathbb{C}\)-module generated by \(\{\hat{f}_1, \ldots, \hat{f}_n, J \hat{f}_1, \ldots, J \hat{f}_n\}\), \(\tilde{N}^{+}\) is a \(\mathbb{C}\)-module generated by \(\{\hat{e}_1, \ldots, \hat{e}_n, J \hat{e}_1, \ldots, J \hat{e}_n\}\), and \(K\) is generated by \(\mathbb{H} \otimes_{\mathbb{C}} H = H \oplus J H\). Thus, \(u\) can be written as \(u = e_i + J e_i + h + J h + f_i + J f_i = 0\) where \(i = 1, \ldots, n\). According to Definition 2.9, we obtain
\begin{equation}
    \begin{split}
        u \cdot 1 &= e_i \cdot 1 + J e_i \cdot 1 + h \cdot 1 + J h \cdot 1 + f_i \cdot 1 + J f_i \cdot 1 \\&= \langle \lambda \cdot h \rangle + \langle \lambda \cdot J h \rangle + v_i + J v_i \\&= 0
    \end{split}
\end{equation}
Thus, it follows that
\begin{equation}
    \langle \lambda \cdot h \rangle + \langle \lambda \cdot J h \rangle = 0
\end{equation}
for every \(\lambda \in H^*\). Hence, \(h = J h = 0\). Furthermore, we also obtain \(v_i + J v_i = 0\). Since \(v_i\) and \(J v_i\) are linearly independent, it follows that \(v_i = J v_i = 0\). Therefore, part (c) is proven. It is also noted that \(\tilde{N}^{+}\) is merely generated by \(\{\hat{e}_1, \ldots, \hat{e}_n, J \hat{e}_1, \ldots, J \hat{e}_n\}\) together with all possible double commutators of these elements and their linear combinations. The subalgebra \(\tilde{N}^{-}\) has the same structure as \(\hat{e}_i\) replaced by \(\hat{f}_i\) and \(J \hat{e}_i\) replaced with \(J \hat{f}_i\). Part (d) is a direct consequence of (b) and (c). Next, we will show this by considering homogeneous elements in \(\hat{N}^+\), namely the multiple commutators of \(\epsilon_i\) (denoting \(e_i\) or \(J e_i\)) where \(\epsilon_1\) appears \(k_1\) times, \(\epsilon_2\) appears \(k_2\) times, and so on. Furthermore, these elements are expressed as
\begin{equation}
    [\epsilon_1, \epsilon_1, \ldots, \epsilon_1, \epsilon_2, \epsilon_2, \ldots, \epsilon_n, \epsilon_n]
\end{equation}
Now, using the basic commutation relation
\begin{equation}
    \text{ad} \, h(\epsilon_j) = [h, \epsilon_j] = \langle \alpha_j, h \rangle \epsilon_j
\end{equation}
it can be easily proven, using the Jacobi identity and mathematical induction, that
\begin{equation}
    \text{ad} \, h\left([\epsilon_1, \epsilon_1, \ldots, \epsilon_1, \epsilon_2, \epsilon_2, \ldots, \epsilon_n, \epsilon_n]\right) = \langle k_1 \alpha_1 + k_2 \alpha_2 + \cdots + k_n \alpha_n, h \rangle [\epsilon_1, \epsilon_1, \ldots, \epsilon_1, \epsilon_2, \epsilon_2, \ldots, \epsilon_n, \epsilon_n]
\end{equation}
For example, for any \(h \in H_r\), we have
\begin{equation}
    \text{ad} \, h([\epsilon_1, \epsilon_2]) = [h, [\epsilon_1, \epsilon_2]] = -[\epsilon_1, [\epsilon_1, h]] - [\epsilon_2, [h, \epsilon_1]] = \langle \alpha_1 + \alpha_2, h \rangle [\epsilon_1, \epsilon_2]
\end{equation}
This shows that homogeneous elements in \(\hat{N}^+\) are simultaneous eigenvectors of \(\text{ad} \, h\) (\(h \in H_r\)). The eigenvalue is \(\langle \alpha, h \rangle\) where \(\alpha\) is
\begin{equation}
    \alpha = \sum_{i=1}^n k_i \alpha_i \quad (k_i = 0, 1, \ldots; \sum_{i=1}^n k_i \neq 0)
\end{equation}
This is exactly what is meant in (d).

The subspace \(\tilde{g}_\alpha\) with \(\alpha \neq 0\) and \(\alpha \in \mathbb{Q}^+\) is a simultaneous eigenspace of \(\text{ad} \, h\) (\(h \in H_r\)) with eigenvalue \(\langle \alpha, h \rangle\). For any \(\alpha\), it is certainly possible that \(\tilde{g}_\alpha = 0\).

Note that \(\alpha\) can never be in the form \(k \alpha_i\) with \(k \in \mathbb{Z}\) and \(k \neq \pm 1\) because the corresponding Lie algebra elements are
\begin{equation}
    [\epsilon_i, [\epsilon_i, [\ldots, \epsilon_i]] ] = 0.
\end{equation}
The only multiples of \(\alpha_i\) that appear as roots are \(\alpha_i\) and \(-\alpha_i\). 

The dimension of \(\tilde{g}_\alpha\) is finite. This follows from \(\dim \tilde{g}_\alpha \leq n \left| \text{ht} \, \alpha \right|\). Therefore, it can be concluded that Theorem 2.2 is proven.
\end{proof}
\begin{example}
    \begin{math}
        \mathfrak{sl}(n,\mathbb{H})
    \end{math}
    is the Universal Kac-Moody Quaternion Lie algebra.
\end{example}
In the following subsection, the construction of standard quaternion Kac-Moody Lie algebras will be discussed.
\section{Construction of Standard Kac-Moody Quaternion Lie Algebra}
\label{sec3}
In this section, we will discuss the construction of the Standard Quaternionic Kac-Moody Lie algebra. Based on the discussion in section 2, the Universal Kac-Moody Quaternion Lie algebra has been constructed. This construction will be used to build the Standard Quaternionic Kac-Moody Lie algebra by imposing restrictions on \( \tilde{g}(A) \), where \( \tilde{g}(A) \) is the Universal Kac-Moody Quaternion Lie algebra. To obtain these restrictions, we use facts from simple quaternionic Lie algebras, specifically that the length of the root chain is finite. In this way, we obtain the form of the Standard Quaternionic Kac-Moody Lie algebra. The main topic of this subsection is to introduce the definition of the Standard Quaternionic Kac-Moody Lie algebra.

It is known from simple quaternionic Lie algebras that the length of the root chain is finite. The restrictions are written as
\begin{equation}
    (\text{ad } \epsilon_j)^{1 - A_{ji}}(\epsilon_i) = 0, \quad (\text{ad } \phi_j)^{1 - A_{ji}}(\phi_i) = 0 \quad (i \neq j)
\end{equation}

Next, restrictions will be applied to \( \tilde{g}(A) \) so that similar relations hold in the Kac-Moody Lie algebra. To obtain these restrictions, first, an ideal \( K \) is defined in \( \tilde{g}(A) \), generated by the elements on the left-hand side of (157). Then, relation (157) holds in the quotient quaternionic Lie algebra \( \tilde{g}(A)/K \). The important role in constructing the ideal \( K \) involves the subsets \( S^+ \) and \( S^- \). The next section starts with the definition and analysis of these sets.

Consider the two subsets in the algebra \( \tilde{g}(A) \) denoted \( S^+ \) and \( S^- \). These subsets are defined by
\begin{equation}
    S^+ = \{ x_{ij} = (\text{ad } \epsilon_j)^{1 - A_{ji}}(\epsilon_i) \mid i \neq j; \; 1, \ldots, n, \; \epsilon_i = e_i \text{ or } J e_i \}
\end{equation}
\begin{equation}
    S^- = \{ y_{ij} = (\text{ad } \phi_j)^{1 - A_{ji}}(\phi_i) \mid i \neq j; \; 1, \ldots, n, \; \phi_i = f_i \text{ or } J f_i \}
\end{equation}

Note that in the subalgebra \( \tilde{N}^+ \), the ideal \( I^+ \) generated by the subset \( S^+ \) intersects trivially with the abelian subalgebra \( H \) and \( JH \), i.e.,
\begin{equation}
    H \cap I^+ = 0 \text{ and } JH \cap I^+ = 0
\end{equation}

Similarly, in \( \tilde{N}^- \), the ideal \( I^- \) generated by \( S^- \) also intersects trivially with the abelian subalgebra \( H \) and \( JH \), i.e.,
\begin{equation}
    H \cap I^- = 0 \text{ and } JH \cap I^- = 0
\end{equation}
Furthermore, we have
\begin{equation}
    I^+ \cap I^- = 0
\end{equation}

The next step is to show that \( I^+ \) and \( I^- \) are ideals in the algebra \( \tilde{g}(A) \). Recall that it is necessary to consider the quotient of \( \tilde{g}(A) \) by the appropriate ideal so that relations similar to (157) hold. To prove that \( I^+ \) and \( I^- \) are ideals in \( \tilde{g}(A) \), we first need to prove Lemma 3.1.
\begin{lemma}
    Let \(S^+\) and \(S^-\) be subset in \(\Tilde{g}(A)\) defined by (158) and (159). Then 
    \begin{equation}
        \text{ad} \, \epsilon_k(y_{ij})=0 \text{ }(k=1,2\dots,n)
    \end{equation}
    and
    \begin{equation}
        \text{ad} \, \phi_k(y_{ij})=0 \text{ } (k=1,2\dots,n)
    \end{equation}
\end{lemma}
\begin{proof}
    We prove only the first statement of the lemma since the proof of the second one goes along similar lines. We distinguish two cases namely \(k=i\) and \(k\ne i\). For \(k=i\) we prove that
    \begin{equation}
        \text{ad } \epsilon_i(\text{ad }\phi_i)^{(1-A_{ij})}(\phi_j)=0\text{  } (i\ne j) 
    \end{equation}
    In order to show this we can use the commutation relations that apply within the quaternion Lie algebra representation \(\mathfrak{sl}(2, \mathbb{H})\). Since \(\{ e_i, J e_i, f_i, J f_i, \alpha_i^v, J \alpha_i^v \}\) spans a subalgebra of \(\mathfrak{sl}(2, \mathbb{H})\) within \(\tilde{g}(A)\), and because \(\tilde{g}(A)\) is a module of \(\mathfrak{sl}(2, \mathbb{H})\), the linear operators \(\epsilon_i\) and \(\phi_i\) satisfy the commutation relation:
    \begin{equation}
        [\epsilon_i, (\phi_i)^m]=-m(m-1)(\phi_i)^{m-1}+m(\phi_i)^{m-1}\alpha_i^v
    \end{equation}
    (This can be proven using mathematical induction and Jacobi's identity.) Therefore, the adjoint operators \(\text{ad} \, \epsilon_i\) and \(\text{ad} \, \phi_i\) satisfy the following commutation relation:
    \begin{equation}
        [\text{ad }\epsilon_i, (\text{ad }\phi_i)^m]=-m(m-1)(\text{ad }\phi_i)^{m-1}+m(\text{ad }\phi_i)^{m-1}\text{ad }\alpha_i^v
    \end{equation}
    Next, equation (167) will be used to show that equation (165) is satisfied. Since \(i \neq j\), it follows that \((\text{ad} \, \epsilon_i)(\phi_j) = 0\), which is obtained after substituting \(m := (1 - A_{ij})\).
    \begin{equation}
        \text{ad }\epsilon_i(\text{ad }\phi_i)^m(\phi_j)=\{-m(m-1)(\text{ad }\phi_i)^{m-1}+m(\text{ad }\phi_i)^{m-1}\text{ad }\alpha_i^v\}(\phi_j)
    \end{equation}
    To proceed, use equations (127) and (128) in the form:
    \begin{equation}
        \text{ad }\alpha_i^v(\phi_j)=-A_{ij}\phi_j=(m-1)\phi_j
    \end{equation}
    Substitute equation (169) into equation (168) to obtain:
    \begin{equation}
        \text{ad }\epsilon_i(\text{ad }\phi_i)^m(\phi_j)=\{-m(m-1)+m(m-1)\}(\text{ad }\phi_i)^{m-1}(\phi_j)=0
    \end{equation}
    This completes the case \( k = i \). For \( k \neq i \), \( k \neq j \), the elements \(\epsilon_k\), \(\phi_i\), and \(\phi_j\) are commutative. Therefore:
    \begin{equation}\begin{split}
        \text{ad }\epsilon_k(y_{ij})&=(\text{ad }\phi_i)^{1-A_{ij}}\text{ad }\epsilon_k(\phi_j)\\&=(\text{ad }\phi_i)^{1-A_{ij}}[\epsilon_k,\phi_j]\\&=0
        \end{split}
    \end{equation}
    Finally, consider the case \( k \neq i \), \( k = j \). Using the commutation relations between the elements \(\epsilon_j\), \(\phi_i\), and \(\alpha_j^v\), we obtain:
    \begin{equation}
        \begin{split}
        \text{ad} \, \epsilon_j (y_{ij}) &= (\text{ad} \, \phi_i)^{1-A_{ij}} \text{ad} \, \epsilon_j (\phi_j) \\&= (\text{ad} \, \phi_i)^{1-A_{ij}} (\alpha_j^v) \\&= (\text{ad} \, \phi_i)^{1-A_{ij}} [\phi_j, \alpha_j^v] \\&= A_{ji} (\text{ad} \, \phi_i)^{-A_{ij}} (\phi_i)
        \end{split}
    \end{equation}
    For \( A_{ij} \neq 0 \) (i.e., \( i \neq j \)), it follows, according to the definition of the generalized Cartan matrix, that \( A_{ij} \leq 0 \). Therefore, the right-hand side of (172) is a multiple of the commutator of \(\phi_i\), which results in zero. If \( A_{ij} = 0 \), then the transposed element \( A_{ji} = 0 \), and the left-hand side of (172) again results in zero. This completes the proof of Lemma 3.1.
\end{proof}
Now it will be proven that \( I^+ \) is an ideal in \( \tilde{g}(A) \). For this, consider any element \( u \in \tilde{g}(A) \). According to part (c) of Theorem 2.2, this element has the unique decomposition:
\begin{equation}
    u = x + h + J h + y
\end{equation}
where
\begin{equation}
    x \in \tilde{N}^+, \quad h \in H, \quad J h \in J H, \quad y \in \tilde{N}^-
\end{equation}
We need to show that for every \( u \in \tilde{g}_\alpha \),
\begin{equation}
    [u, I^+] = [x, I^+] + [h, I^+] + [J h, I^+] + [y, I^+] \subset I^+
\end{equation}
The first term \([x, I^+]\) is straightforward. This term lies in \( I^+ \) because \( I^+ \) is an ideal in \( \tilde{N}^+ \). For the second term, Jacobi's identity is required. Elements of \( I^+ \) are of the form \([z, x_{ij}]\) with \( z \in \tilde{N}^+ \), and linear combinations of such objects are within \( I^+ \), since \( I^+ \) is an ideal generated by \( S^+ \). Applying Jacobi's identity, we get:
\begin{equation}
    [h, [z, x_{ij}]] = [z, [h, x_{ij}]] + [[h, z], x_{ij}]
\end{equation}
\begin{equation}
    [J h, [z, x_{ij}]] = [z, [J h, x_{ij}]] + [[J h, z], x_{ij}]
\end{equation}
As \( h \) and \( J h \) act diagonally on \( \tilde{g}(A) \), we obtain:
\begin{equation}
    [h, x_{ij}] = \lambda x_{ij}, \quad [h, z] = \mu z \quad (z \in \tilde{g}_\mu)
\end{equation}
\begin{equation}
    [J h, x_{ij}] = J \lambda x_{ij}, \quad [J h, z] = J \mu z \quad (z \in \tilde{g}_{J \mu})
\end{equation}
It follows that both terms are in \( I^+ \). 

The term \([y, I^+]\) has the following form (using Jacobi's identity):
\begin{equation}
    [y, [z, x_{ij}]] = [z, [y, x_{ij}]] + [x_{ij}, [z, y]]
\end{equation}
A simple example of elements satisfying this is obtained by taking \( z \) and \( y \) as generator elements, say \( y = \phi_k \) and \( z = \epsilon_m \). This leads to:
\begin{equation}
    [\phi_k, [\epsilon_m, x_{ij}]] = [\epsilon_m, [\phi_k, x_{ij}]] + [x_{ij}, [\epsilon_m, \phi_k]]
\end{equation}
The first term on the right-hand side is zero according to Lemma 3.1. Since \([\epsilon_m, \phi_k] \in H\) or \([\epsilon_m, \phi_k] \in JH\), the second term on the right-hand side is of the form \([x_{ij}, h] = k x_{ij}\) or \([x_{ij}, J h] = J k x_{ij}\). Therefore, \( I^+ \) and similarly \( I^- \) are ideals in \( \tilde{g}(A) \). The final step is to consider the ideal
\begin{equation}
    K := I^+ + I^-
\end{equation}
It has the property that the intersections with \( H \) and \( JH \) are trivial:
\begin{equation}
    K \cap H = 0 \quad \text{and} \quad K \cap JH = 0
\end{equation}
Now the necessary properties have been obtained to define the standard quaternion Kac-Moody Lie algebra.
\begin{definition}
    Let \(\Tilde{g}(A)\) be the Universal Kac-Moody Quaternion Lie algebra of Definition 2.10 and \(K\) the ideal defined in (182). Then the quotient algebra
    \begin{equation}\label{3}
        g(A)=\Tilde{g}(A)/K
    \end{equation} 
    is by definition the \textbf{Standard Kac-Moody Quaternion Lie Algebra} corresponding to the generalized Cartan matrix \(A\).
\end{definition}
\begin{theorem}
    Let \(g(A)\) be the Standard Kac-Moody Quaternion Lie algebra belonging to the generalized Cartan matrix \(A=(A_{ij})\). Then \(g(A)\) has the following properties:
    \begin{enumerate}
        \item Let \(N^-\) be the \(\mathbb{C}\)-module generated by \(\{f_1,\dots,f_n,Jf_1,\dots,Jf_n\}\) and \(N^+\) be the \(\mathbb{C}\)-module generated by \(\{e_1,\dots,e_n,Je_1,\dots,Je_n\}\), and Let \(K\) generated by \(\textbf{H}\otimes_{\textbf{C}} H=H\oplus JH\). Then \(N^-,K, \text{and } N^+\) viewed as a real Lie subalgebra of \(g(A)\) give the decomposition of \(g(A):\) 
        \begin{equation}
            g(A)=N^-\oplus K\oplus N^+
        \end{equation}
        \item Considering \(g(A)\) as an ad \(H_r\) module, where \(H_r\) is a maximal commutative subalgebra of \(g(A)\) we have the root decomposition 
        \begin{equation}
           g(A):=\left(\bigoplus_{\alpha\ne0, \alpha\in Q^+}g_{-\alpha}\right)\oplus g_0\oplus \left(\bigoplus_{\alpha\ne0, \alpha\in Q^+}g_{+\alpha}\right) 
        \end{equation}
        where \(g_\alpha=\{x\in g(A)|\forall h\in H:\text{ad } h(x)=\langle \alpha,h\rangle x\}, H_r\subset g_0=K\). Furthermore, dim \(g_\alpha<\infty\) and \(g_\alpha\subset N^\pm\) for \(\pm \alpha \in Q^+, \alpha\ne0\)
        \item In the Standard Kac-Moody Quaternion Lie algebra \(g(A)\) the following commutation relations hold
        \begin{equation}
            [h,e_i]=\langle \alpha_i,h\rangle e_i, [Jh,e_i]=\langle \alpha_i,Jh\rangle e_i, [h,Je_i]=\langle \alpha_i,Jh\rangle e_i,
        \end{equation}
        \begin{equation}
            [Jh,Je_i]=-\langle \alpha_i,h\rangle e_i, [h,f_i]=-\langle \alpha_i,f_i\rangle,  [Jh,f_i]=-\langle \alpha_i,Jh\rangle f_i,
        \end{equation}
        \begin{equation}
            [h,Jf_i]=-\langle \alpha_i,h\rangle Jf_i, [Jh,Jf_i]=\langle \alpha_i,h\rangle f_i,
        \end{equation}
        \begin{equation}
            [e_i,f_j]=\delta_{ij} \alpha_i^v,[Je_i,f_j]=\delta_{ij} J\alpha_i^v,[e_i,Jf_j]=\delta_{ij} J\alpha_i^v,
        \end{equation}
        \begin{equation}
           [Je_i,Jf_j]=\delta_{ij} J\alpha_i^v,[h,h]=0,[h,Jh]=0,[Jh,Jh]=0 
        \end{equation}
        \item Furthermore, \(g(A)\) one has
        \begin{equation}
            (\text{ad} \epsilon_j)^{1-A_{ij}}(\epsilon_i)=0 \text{ } \{i\ne j;i,j=1,\dots,n\}
        \end{equation}
        \begin{equation}
            (\text{ad} \epsilon_j)^{1-A_{ij}}(\phi_i)=0 \text{ } \{i\ne j;i,j=1,\dots,n\}
        \end{equation}
        where \(\phi_i=e_i \text{ or } Je_i\) and \(\phi_i=f_i \text{ or } Jf_i\).
    \end{enumerate}
\end{theorem}
\renewcommand\qedsymbol{$\blacksquare$}
\begin{proof}
    To prove this theorem, it is sufficient to show that \(H, JH\), and \({e_i,Je_i,f_i, Jf_i}_{i=1}^n\) satisfy the canonical projection \(\Psi\) from \(\Tilde{g}(A)\) to \(g(A)=\Tilde{g}(A)/\varphi\). The subsequent theorem will be clear from Theorem 2.2.
    
    For \(H\) and \(JH\), this is straightforward because \(\varphi\cap H=0\) and \(\varphi\cap JH=0\). Therefore, \(H\) and \(JH\) are mapped one to one by \(\Psi\). The same reasoning applies to \({e_i,Je_i,f_i, Jf_i}_{i=1}^n\) as in the proof of Theorem 2.2.
\end{proof}
\begin{example}
    \begin{math}
        \mathfrak{sl}(n,\mathbb{H})
    \end{math}
    is the Standard Kac-Moody Quaternion Lie algebra.
\end{example}
In the following subsection, the construction of reduced quaternion Kac-Moody Lie algebras will be discussed.
\section{Construction of Reduced Kac-Moody Quaternion Lie Algebra}
\label{sec4}
In this section, we will discuss the construction of reduced quaternionic Kac-Moody Lie algebras. Based on the discussion in section 2.1, a universal Kac-Moody quaternion Lie algebra has been constructed. This construction will then be used to construct the reduced quaternionic Kac-Moody Lie algebra by creating a quaternionic Lie algebra quotient. Specifically, let \( \mathfrak{g}(A) \) be the universal Kac-Moody quaternion Lie algebra and let \( \mathfrak{v} \) be an ideal in \( \mathfrak{g}(A) \). The quotient Lie algebra \( \mathfrak{g}(A) / \mathfrak{v} \) represents the reduced quaternionic Kac-Moody Lie algebra. The main topic of this subsection is to introduce the definition of the reduced quaternionic Kac-Moody Lie algebra.

Before defining the reduced quaternionic Kac-Moody Lie algebra, we will first explain the concept of gradation and one lemma that will be used to prove the theorem resulting from the definition of the reduced quaternionic Kac-Moody Lie algebra.

Let \( M \) be an abelian group. The decomposition \( V = \bigoplus_{a \in M} V_a \) of the quaternionic module \( (\mathcal{V}, J) \) into a direct sum of its submodules is called the \( M \)-gradation of \( (\mathcal{V}, J) \). A submodule \( U \subset V \) is called graded if \( U = \bigoplus_{a \in M} (U \cap V_a) \). Elements of \( V_a \) are called homogeneous of degree \( \alpha \). The next step is to prove Lemma 4.1, which will be used later in the proof of Proposition 4.2.
\begin{lemma}\label{Lemma1}
    Let \(H_r\) be a commutative quaternion Lie algebra and \(V\) be a diagonalisable \(H_r\)-module, i.e.
    \begin{equation}
        V=\bigoplus_{\lambda\in H_r^*} V_\lambda \text{   dimana  } V_\lambda=\{v\in V|h(v)=\langle\lambda,h\rangle v, \forall h\in H_r\}
    \end{equation}
    Let \(U\) be a submodule of \(V\). Then
    \begin{equation}
        U=\bigoplus_{\lambda\in H_r^*} (U\cap V_\lambda)
    \end{equation}
\end{lemma}
\begin{proof}
    Let \( v \in U \) be decomposed as \( v = \sum_{i=1}^m v_{\lambda_i} \) where \( v_{\lambda_i} \in V_{\lambda_i} \). We will then prove that \( v_{\lambda_i} \in U \) for each \( i \in \{1, \ldots, m\} \). Suppose \( h \in H_r \) such that each \( \langle \lambda_i, h \rangle \) is distinct. Thus, we have \(h^k(v) = \sum_{i=1}^m \langle \lambda_i, h \rangle^k v_{\lambda_i} \in U\) for \( k \in \{0, \ldots, m - 1\} \) because \( U \) is an \( H_r \)-module. This is a system of linear equations with a non-degenerate matrix. Hence, \( v_{\lambda_i} \in U \) for each \( i \in \{1, \ldots, m\} \).
\end{proof}
In addition, one more proposition is needed, which will be used to define the reduced quaternionic Kac-Moody Lie algebra.
\begin{prop}
    Let \(\Tilde{g}(A)\) be the Universal Kac-Moody Quaternion Lie algebra. Then, among the ideals of \(\Tilde{g}(A)\) intersecting \(H\) and \(JH\) trivially, there exist a unique maximal ideal \(\vartheta\). Futhermore, 
    \begin{equation}
        \vartheta=(\vartheta\cap\Tilde{N}^-)\oplus(\vartheta\cap\Tilde{N}^+)
    \end{equation}
\end{prop}
\begin{proof}
    Recall from Lemma \ref{Lemma1} that every ideal \(i\) of \(\Tilde{g}(A)\) is graded (for the \(Q\)-gradation, i.e.
    \begin{equation}
        i=\bigoplus_{\alpha\in Q} ( i\cap \Tilde{g}_\alpha)
    \end{equation}
    in particular, a sum of ideals of \(\Tilde{g}(A)\) intersecting \(H\) and \(JH\) trivially also intersect \(H\) and \(JH\) trivially, and we let \(\vartheta\) be the unique maximal such ideal. Note that
    \begin{equation}
        [e_i,\Tilde{N}^+]\subset \Tilde{N}^+ \text{ maka } [e_i, \vartheta\cap\Tilde{N}^+]\subset \vartheta\cap\Tilde{N}^+
    \end{equation}
    \begin{equation}
        [Je_i,\Tilde{N}^+]\subset \Tilde{N}^+ \text{ maka } [Je_i, \vartheta\cap\Tilde{N}^+]\subset \vartheta\cap\Tilde{N}^+
    \end{equation}
    \begin{equation}
        [h_i,\Tilde{N}^+]\subset \Tilde{N}^+ \text{ maka } [h_i, \vartheta\cap\Tilde{N}^+]\subset \vartheta\cap\Tilde{N}^+
    \end{equation}
    \begin{equation}
        [Jh_i,\Tilde{N}^+]\subset \Tilde{N}^+ \text{ maka } [Jh_i, \vartheta\cap\Tilde{N}^+]\subset \vartheta\cap\Tilde{N}^+
    \end{equation}
    \begin{equation}
        [f_i,\Tilde{N}^+]\subset \Tilde{N}^+\oplus h\oplus Jh \text{ maka } [f_i, \vartheta\cap\Tilde{N}^+]\subset \vartheta\cap(\Tilde{N}^+\oplus h\oplus Jh)\subset \vartheta\cap\Tilde{N}^+
    \end{equation}
    \begin{equation}
        [Jf_i,\Tilde{N}^+]\subset \Tilde{N}^+\oplus h\oplus Jh \text{ maka } [Jf_i, \vartheta\cap\Tilde{N}^+]\subset \vartheta\cap(\Tilde{N}^+\oplus h\oplus Jh)\subset \vartheta\cap\Tilde{N}^+
    \end{equation}
    and hence \([\Tilde{g}(A), \vartheta\cap\Tilde{N}^+]\subset \vartheta\cap\Tilde{N}^+\), that is, \(\vartheta\cap\Tilde{N}^+\) is an ideal in \(\Tilde{g}(A)\) (and similarly for \(\vartheta\cap\Tilde{N}^-\). This shows that \(\vartheta=(\vartheta\cap\Tilde{N}^+)\oplus(\vartheta\cap\Tilde{N}^-)\).
\end{proof}
Now, the necessary properties have been obtained to define the reduced quaternionic Kac-Moody Lie algebra.
\begin{definition}
    Let \(\Tilde{g}(A)\) be the Universal Kac-Moody Quaternion Lie algebra of Definition 2.10 and \(\vartheta\) the ideal defined in (196). Then the quotient algebra
    \begin{equation}\label{4}
        \Breve{g}(A)=\Tilde{g}(A)/\vartheta
    \end{equation} 
    is by definition the \textbf{Reduced Kac-Moody Quaternion Lie Algebra} corresponding to the generalized Cartan matrix \(A\).
\end{definition}
\begin{theorem}
    Let \(\Breve{g}(A)\) be the Reduced Kac-Moody Quaternion Lie algebra belonging to the generalized Cartan matrix \(A=(A_{ij})\). Then \(\Breve{g}(A)\) has the following properties:
    \begin{enumerate}
        \item Let \(\Breve{N}^-\) be the \(\mathbb{C}\)-module generated by \(\{f_1,\dots,f_n,Jf_1,\dots,Jf_n\}\) and \(\Breve{N}^+\) be the \(\mathbb{C}\)-module generated by \(\{e_1,\dots,e_n,Je_1,\dots,Je_n\}\), and Let \(K\) generated by \(\textbf{H}\otimes_{\textbf{C}} H=H\oplus JH\). Then \(\Breve{N}^-,K, \text{and } \Breve{N}^+\) viewed as a real Lie subalgebra of \(\Breve{g}(A)\) give the decomposition of \(\Breve{g}(A):\) 
        \begin{equation}
            \Breve{g}(A)=\Breve{N}^-\oplus K\oplus \Breve{N}^+
        \end{equation}
        \item Considering \(\Breve{g}(A)\) as an ad \(H_r\) module, where \(H_r\) is a maximal commutative subalgebra of \(\Breve{g}(A)\) we have the root decomposition 
        \begin{equation}
            \Breve{g}(A):=\left(\bigoplus_{\alpha\ne0, \alpha\in Q^+}\Breve{g}_{-\alpha}\right)\oplus \Breve{g}_0\oplus \left(\bigoplus_{\alpha\ne0, \alpha\in Q^+}\Breve{g}_{+\alpha}\right)
        \end{equation}
        where \(\Breve{g}_\alpha=\{x\in \Breve{g}(A)|\forall h\in H:\text{ad } h(x)=\langle \alpha,h\rangle x\}, H_r\subset \Breve{g}_0=K\). Furthermore, dim \(\Breve{g}_\alpha<\infty\) and \(\Breve{g}_\alpha\subset N^\pm\) for \(\pm \alpha \in Q^+, \alpha\ne0\)
        \item In the Rudeced Kac-Moody Quaternion Lie algebra \(\Breve{g}(A)\) the following commutation relations hold
        \begin{equation}
            [h,e_i]=\langle \alpha_i,h\rangle e_i, [Jh,e_i]=\langle \alpha_i,Jh\rangle e_i, [h,Je_i]=\langle \alpha_i,Jh\rangle e_i,
        \end{equation}
        \begin{equation}
            [Jh,Je_i]=-\langle \alpha_i,h\rangle e_i, [h,f_i]=-\langle \alpha_i,f_i\rangle,  [Jh,f_i]=-\langle \alpha_i,Jh\rangle f_i,
        \end{equation}
        \begin{equation}
            [h,Jf_i]=-\langle \alpha_i,h\rangle Jf_i, [Jh,Jf_i]=\langle \alpha_i,h\rangle f_i,
        \end{equation}
        \begin{equation}
            [e_i,f_j]=\delta_{ij} \alpha_i^v,[Je_i,f_j]=\delta_{ij} J\alpha_i^v,[e_i,Jf_j]=\delta_{ij} J\alpha_i^v,
        \end{equation}
        \begin{equation}
          [Je_i,Jf_j]=\delta_{ij} J\alpha_i^v,[h,h]=0,[h,Jh]=0,[Jh,Jh]=0 
        \end{equation}
        \item \(\Breve{g}(A)\) has no nozero ideal \(\vartheta\) such that \(\vartheta\cap H=\textbf{0}\) and \(\vartheta\cap JH=\textbf{0}\)
    \end{enumerate}
\end{theorem}
\renewcommand\qedsymbol{$\blacksquare$}
\begin{proof}
    To prove this theorem, it is sufficient to show that \(H, JH\), and \({e_i,Je_i,f_i, Jf_i}_{i=1}^n\) satisfy the canonical projection \(\Psi\) from \(\Tilde{g}(A)\) to \(\Breve{g}(A)=\Tilde{g}(A)/\varphi\). The subsequent theorem will be clear from Theorem 2.2.
    
    For \(H\) and \(JH\), this is straightforward because \(\varphi\cap H=0\) and \(\varphi\cap JH=0\). Therefore, \(H\) and \(JH\) are mapped one to one by \(\Psi\). The same reasoning applies to \({e_i,Je_i,f_i, Jf_i}_{i=1}^n\) as in the proof of Theorem 2.2.
\end{proof}
\begin{example}
    \begin{math}
        \mathfrak{sl}(n,\mathbb{H})
    \end{math}
    is the Reduced Kac-Moody Quaternion Lie algebra.
\end{example}

\end{document}